\documentclass[11pt]{elsarticle}

\usepackage{amsmath}
\usepackage{amssymb}
\usepackage{amsthm}

\newtheorem{theorem}{Theorem}[section]
\newtheorem{lemma}[theorem]{Lemma}
\newtheorem{proposition}[theorem]{Proposition}
\newtheorem{corollary}[theorem]{Corollary}
\newtheorem{question}[theorem]{Question}
\theoremstyle{definition}
\newtheorem{example}[theorem]{Example}
\newtheorem{remark}[theorem]{Remark}
\theoremstyle{plain}
\newtheorem{mainthm}{Theorem}

\newtheorem{maincor}[mainthm]{Corollary}

\newcommand{\GF}{\mathrm{GF}}
\newcommand{\Gal}{\mathrm{Gal}}
\newcommand{\Aut}{\mathrm{Aut}}
\newcommand{\End}{\mathrm{End}}
\newcommand{\Syl}{\mathrm{Syl}}
\newcommand{\GL}{\mathrm{GL}}
\newcommand{\Z}{\mathbb{Z}}
\newcommand{\nsg}{\trianglelefteq}

\begin{document}

\begin{frontmatter}

\title{Abelian quotients and the largest orbit sizes of linear groups of odd order}

\author{Thomas Michael Keller}
\ead{keller@txstate.edu}

\address{Department of Mathematics, Texas State University, 601 University Drive, San Marcos, TX 78666, United States of America}

\begin{abstract}
Let $G$ be a finite nonabelian group of odd order, let $p$ be the smallest prime divisor of $|G|$, and let $V$ be a finite faithful completely reducible $G$-module, possibly of mixed characteristic. Suppose that $M$ is the largest orbit size in the action of $G$ on $V$. It is known that $|G:G'|\le M$, and that equality holds only for abelian groups if $|G|$ is odd. We prove that $|G:G'|\le M/p$, unless $p=3$ and $G$ is the direct product of an abelian group and a group isomorphic to $\Gamma(2^3)$ or to $\Gamma(2^3)\times C_3$ acting on $V$ in a specific way, in which case $|G:G'|=3M/7$. In particular, $|G:G'|\le 3M/7$ for every nonabelian group of odd order. If $|V|$ is odd, then $|G:G'|=M/p$ holds if and only if $|G'|=p$; more precisely, nilpotent groups then have regular orbits, whereas non-nilpotent groups satisfy $|G:G'|\le 3M/13$, and even $|G:G'|<M/(2p)$ if $p\ge5$. The bound $3M/13$ is best possible.
\end{abstract}

\begin{keyword}
Finite linear groups \sep maximal orbit size \sep groups of odd order \sep regular orbits \sep semilinear groups \sep commutator subgroups
\MSC[2020] 20D10
\end{keyword}

\end{frontmatter}

\section{Introduction}\label{sec:intro}

One important aspect of finite linear groups is the study of their orbits, and maximal orbit sizes play an important role in understanding the structure of a finite group action. Let $G$ be a finite group, let $V$ be a finite faithful completely reducible $G$-module, possibly of mixed characteristic, and let $M$ be the largest orbit size in the action of $G$ on $V$. Yang and the author proved that
\[
|G:G'|\le M,
\]
first for solvable groups \cite{KY16} and then for arbitrary finite groups \cite{KY20}. Moreover, if $G$ is nonabelian and $|G:G'|=M$, then $G$ is nilpotent and has at least two orbits of size $M$ on $V$ (see \cite{QY21} and \cite[Theorem 2.1]{CK24b}). This extreme case was studied in \cite{CK24a,JK20,KP22}, and it was completely classified by \c{C}\i{}narc\i{} and the author in \cite{CK24b}: if $G$ is nonabelian and $|G:G'|=M$, then $G$ is the direct product of an abelian group and of copies of $D_8$, $D_8\circ C_4$ and $D_8\circ Q_8$, and $V$ is the direct sum of a faithful module for the abelian factor and of corresponding copies of $V(2,3)$, $V(2,5)$ and $V(4,3)$ \cite[Theorems A and B]{CK24b}. The irreducible actions occurring here are exactly the ones in Passman's classification of the linear actions without regular orbits of groups of type $E(p,m)$ \cite[Lemma 1.5(i)]{Pa69}.

All groups in this classification have nonabelian Sylow $2$-subgroups. Consequently, if $|G|$ is odd, then $|G:G'|=M$ holds if and only if $G$ is abelian (Corollary~\ref{collapse}). This raises the question what the best possible bound for $|G:G'|$ in terms of $M$ is if $G$ is a nonabelian group of odd order, and which group actions attain it. It turns out that the answer is governed by the smallest prime divisor $p$ of $|G|$. Indeed, if $G$ is a group of odd order with $|G'|=p$, then $G$ is nilpotent and has a regular orbit on every finite faithful completely reducible $G$-module (Proposition~\ref{Gprime}), so that $|G:G'|=M/p$. Our main result shows that, apart from two small exceptional configurations for $p=3$, the bound $M/p$ holds in general.

To describe the exceptions, let $D_1=\Gamma(2^3)$, acting naturally on $U_1=\GF(8)$. Let $A_0$ be the subgroup of order $21$ of $\GF(64)^\times$, let $\sigma_0\in\Gal(\GF(64))$ be the automorphism $x\mapsto x^4$, and let $D_2=A_0\langle\sigma_0\rangle\le\Gamma(2^6)$, acting naturally on $U_2=\GF(64)$. Then $D_2\cong\Gamma(2^3)\times C_3$, and $U_2$ can be viewed as $U_1\otimes_{\GF(2)}\GF(4)$, where $\Gamma(2^3)$ acts on the first factor and $C_3=\GF(4)^\times$ acts by multiplication on the second factor (Lemma~\ref{exc}). We say that a pair $(G,V)$ is \emph{exceptional} if $G=D\times A$ and $V=U\oplus W$, where $(D,U)$ is isomorphic to $(D_1,U_1)$ or to $(D_2,U_2)$ as a linear group, $A$ is abelian, $D$ acts trivially on $W$, and $A$ acts trivially on $U$ and faithfully on $W$. Note that then $3$ divides $|G|$ and $|V|$ is even.

\begin{mainthm}\label{thmA}
Let $G$ be a finite nonabelian group of odd order, let $p$ be the smallest prime divisor of $|G|$, and let $V$ be a finite faithful completely reducible $G$-module, possibly of mixed characteristic. Suppose that $M$ is the largest orbit size in the action of $G$ on $V$. Then
\[
|G:G'|\le\frac{M}{p},
\]
unless $(G,V)$ is exceptional, in which case $p=3$ and $|G:G'|=\frac37M$. Moreover, if $|G'|=p$, then $|G:G'|=M/p$.
\end{mainthm}

Since $p\ge3$, we obtain the following best possible bound for arbitrary groups of odd order.

\begin{maincor}\label{corB}
Let $G$ be a finite nonabelian group of odd order and let $V$ be a finite faithful completely reducible $G$-module, possibly of mixed characteristic. If $M$ is the largest orbit size in the action of $G$ on $V$, then
\[
|G:G'|\le\frac37M,
\]
and equality holds if and only if $(G,V)$ is exceptional. Otherwise, $|G:G'|\le M/3$.
\end{maincor}

In particular, the quotient $M/|G:G'|$ never lies strictly between $7/3$ and $3$ for a nonabelian group of odd order. The exceptional configurations are due to the arithmetic of small semilinear groups in characteristic $2$. For an odd prime $p$, the smallest prime divisor of the order of $\Gamma(2^p)\cong C_{2^p-1}\rtimes C_p$ is $p$, and acting on $V(p,2)$ this group satisfies $|G:G'|=pM/(2^p-1)$ (Example~\ref{ex:gamma}). This exceeds $M/p$ only for $p=3$, where $\Gamma(2^3)=D_1$.

If $|V|$ is odd, then $(G,V)$ is never exceptional, and much more can be said.

\begin{mainthm}\label{thmC}
Let $G$ be a finite nonabelian group of odd order, let $p$ be the smallest prime divisor of $|G|$, and let $V$ be a finite faithful completely reducible $G$-module, possibly of mixed characteristic, such that $|V|$ is odd. Suppose that $M$ is the largest orbit size in the action of $G$ on $V$.
\begin{itemize}
\item[(a)] If $G$ is nilpotent, then $G$ has a regular orbit on $V$. In particular, $M=|G|$ and $|G:G'|=M/|G'|$.
\item[(b)] If $G$ is not nilpotent, then there exists $v\in V$ such that $|C_G(v)|\le\frac{3}{13}|G'|$, and if $p\ge5$, then there exists $v\in V$ such that $|C_G(v)|<\frac{1}{2p}|G'|$. In particular,
\[
|G:G'|\le\frac{3}{13}M,\qquad\text{and}\qquad |G:G'|<\frac{M}{2p}\ \text{ if } p\ge5.
\]
\item[(c)] We have $|G:G'|\le M/p$, and equality holds if and only if $|G'|=p$.
\end{itemize}
\end{mainthm}

The constant $3/13$ in Theorem~\ref{thmC}(b) is best possible: the Frobenius group $C_{13}\rtimes C_3\le\Gamma(3^3)$ acting on $V(3,3)$ has no regular orbit and satisfies $|G:G'|=3M/13$ (Example~\ref{ex:13}). For $p\ge5$, the bound $M/(2p)$ is close to best possible: if $2p+1$ is a prime, then the Frobenius group $C_{2p+1}\rtimes C_p\le\Gamma(3^p)$ acting on $V(p,3)$ has a regular orbit and satisfies $|G:G'|=M/(2p+1)$ (Example~\ref{ex:SG}).

The groups attaining the bound $M/p$ for all modules can be described as follows.

\begin{maincor}\label{corD}
Let $G$ be a finite group of odd order and let $p$ be the smallest prime divisor of $|G|$. Then $|G'|=p$ if and only if $G=P\times A$, where $P$ is a $p$-group with $|P'|=p$ and $A$ is an abelian group of order prime to $p$. In this case $P'\le Z(G)$, the group $P/Z(P)$ is elementary abelian of order $p^{2m}$ for some $m\ge1$, $G$ has a regular orbit on every finite faithful completely reducible $G$-module $V$, possibly of mixed characteristic, and
\[
|G:G'|=\frac{M}{p}
\]
for every such $V$, where $M$ denotes the largest orbit size of $G$ on $V$. If $G$ has a faithful irreducible module over a finite field, then $Z(G)$, and hence also $A$, is cyclic.
\end{maincor}

Examples of $p$-groups $P$ as in Corollary~\ref{corD} are the extraspecial groups of order $p^{1+2m}$, their central products with cyclic $p$-groups, and the groups $C_{p^{c+1}}\rtimes C_p$ with $c\ge1$, where a generator of $C_p$ acts by raising to the power $1+p^c$. For instance, the Heisenberg group of order $p^3$ acting on $V(p,s)$ for a prime power $s\equiv1\pmod p$ satisfies $|G:G'|=M/p$ (Example~\ref{ex:heis}).

We emphasize that, in contrast to \cite[Theorem A]{CK24b}, the characterization of equality in Theorem~\ref{thmC}(c) does not impose any restriction on the module $V$. The reason is that in this situation nilpotent groups always have regular orbits (Theorem~\ref{Yang} below; for groups of type $E(p,m)$ with $p$ odd see also \cite[Lemma 1.5(i)]{Pa69}), so that extremality is caused solely by the commutator subgroup being as small as possible, whereas in \cite{CK24b} it is caused by the absence of regular orbits. The hypothesis that $|V|$ is odd cannot be omitted in Theorem~\ref{thmC}(c), even for nilpotent groups: the group $C_3\wr C_3$ acting monomially on $V(3,4)$ satisfies $|G:G'|=M/3$, although $|G'|=9$ (Example~\ref{ex:wr}). We do not know whether such examples exist for $p\ge5$ (Question~\ref{q1}).

The proofs of Theorems~\ref{thmA} and~\ref{thmC} use different methods. For Theorem~\ref{thmA}, we use the inequality $|G:G'|\le|F(G):F(G)'|$ of \cite{KY20} (Theorem~\ref{KYF} below) and the nilpotent case to reduce to the situation that $F(G)$ is abelian and every regular orbit of $F(G)$ on $V$ is $G$-invariant. In this situation, either $V_{F(G)}$ is homogeneous, and then $G$ is a semilinear group which can be treated directly (Section~\ref{sec:semilinear}), or $V$ has characteristic $2$ and the orbits of $G$ are large (Section~\ref{sec:invariant}). The refinement for $|V|$ odd in Theorem~\ref{thmC} is based on the elementary observation that $v$ and $-v$ never lie in the same orbit of a group of odd order acting on a module of odd characteristic (Lemma~\ref{minus}). Together with a theorem of Gluck \cite{Gl83}, this allows us to control imprimitive modules, while quasiprimitive modules are handled by a result of Espuelas \cite{Es91} and the analysis of semilinear groups (Section~\ref{sec:oddV}).

We have adopted standard notation. All groups considered in this paper are finite, and groups act on vector spaces from the right. Extending a notion from analytic number theory to finite groups, if $p$ is a prime, we call a group \emph{$p$-rough} if every prime divisor of its order is at least $p$. Thus the $3$-rough groups are exactly the groups of odd order, and subgroups and quotients of $p$-rough groups are $p$-rough; in the situation of our main results, $G$ is $p$-rough. If a group $X$ acts on a finite set $\Omega$ and $\omega\in\Omega$, we write $\omega^X$ for the $X$-orbit of $\omega$, $C_X(\omega)$ for the stabilizer of $\omega$ in $X$, and $M(X,\Omega)$ for the largest size of an $X$-orbit on $\Omega$; the orbit $\omega^X$ is regular if $C_X(\omega)=1$. If $X$ acts faithfully and completely reducibly on a finite vector space $V$, we put
\[
\rho(X,V)=\frac{M(X,V)}{|X:X'|},
\]
so that $\rho(X,V)\ge1$ by Theorem~\ref{KY} below. For a vector space $V$ we put $V^\#=V\setminus\{0\}$, and for $X\le\GL(V)$ we denote by $C_V(X)$ the space of fixed points of $X$ on $V$. As in \cite{MW93}, a completely reducible $G$-module $V$ is allowed to be of mixed characteristic, i.e., $V=V_1\oplus\cdots\oplus V_t$, where each $V_i$ is an irreducible $G$-module over a finite field of characteristic $p_i$; thus $|V|$ is odd if and only if all $p_i$ are odd. A faithful irreducible $G$-module $V$ is quasiprimitive if $V_N$ is homogeneous for every $N\nsg G$. For a finite field $E$ we write $\Gal(E)$ for the group of all field automorphisms of $E$,
\[
\Gamma(E)=\{x\mapsto ax^\sigma\mid a\in E^\times,\ \sigma\in\Gal(E)\}
\]
for the semilinear group of $E$, and $\Gamma_0(E)=\{x\mapsto ax\mid a\in E^\times\}\cong E^\times$; we also write $\Gamma(q^n)=\Gamma(\GF(q^n))$. As in \cite{CK24b}, $V(n,q)$ denotes the $n$-dimensional vector space over the field with $q$ elements. Two linear groups $X\le\GL(V)$ and $Y\le\GL(W)$ are called isomorphic if there is an additive bijection $\varphi\colon V\to W$ with $\varphi^{-1}X\varphi=Y$. For $x,y$ in a group we write $[x,y]=x^{-1}y^{-1}xy$, and if a group $X$ acts on an abelian group $B$, we write $[B,X]$ for the subgroup generated by all $b^{-1}b^x$ with $b\in B$ and $x\in X$. Finally, by the Feit--Thompson theorem \cite{FT63}, groups of odd order are solvable; this will be used without further reference.

\section{Auxiliary results}\label{sec:aux}

We begin with some results from the literature.

\begin{theorem}[{\cite[Theorem 1.1]{KY16}, \cite[Theorem 1.1]{KY20}}]\label{KY}
Let $G$ be a finite group and let $V$ be a finite faithful completely reducible $G$-module, possibly of mixed characteristic. Then $|G:G'|\le M(G,V)$.
\end{theorem}

\begin{theorem}[{\cite[Theorem 3.1]{CK24b}}]\label{CK}
Let $G$ be a finite nonabelian group and let $V$ be a finite faithful completely reducible $G$-module, possibly of mixed characteristic. If $|G:G'|=M(G,V)$, then $G=P\times A$, where $P$ is the Sylow $2$-subgroup of $G$ and $A$ is an abelian group of odd order.
\end{theorem}

\begin{lemma}[{\cite[Lemma 2.1]{KY16}}]\label{quot}
Let $G$ be a finite group and let $N\nsg G$. Then
\[
|G:G'|=|G:G'N|\cdot|N:N\cap G'|,
\]
and $|G:G'|$ divides $|G/N:(G/N)'|\cdot|N:N'|$.
\end{lemma}

\begin{theorem}[{\cite[Theorem 2.2]{KY20}}]\label{KYF}
Let $G$ be a finite group such that $F^*(G)=F(G)$; for instance, let $G$ be solvable. Then
\[
|G:G'|\le|F(G):F(G)'|.
\]
\end{theorem}

\begin{lemma}\label{product}
Let the group $G$ act linearly on the finite vector space $V=U\oplus W$, where $U$ and $W$ are $G$-invariant subspaces. Then
\[
M(G,V)\ge M(G,U)\cdot M(C_G(U),W).
\]
\end{lemma}

\begin{proof}
Choose $u\in U$ with $|u^G|=M(G,U)$ and $w\in W$ with $|w^{C_G(U)}|=M(C_G(U),W)$. As $U$ and $W$ are $G$-invariant, $C_G(u+w)=C_G(u)\cap C_G(w)$, and hence
\[
|(u+w)^G|=|G:C_G(u)|\cdot|C_G(u):C_G(u)\cap C_G(w)|=|u^G|\cdot|w^{C_G(u)}|.
\]
Since $C_G(U)\le C_G(u)$, we have $w^{C_G(U)}\subseteq w^{C_G(u)}$, and the assertion follows.
\end{proof}

\begin{lemma}\label{comm}
Let $p$ be an odd prime and let $G$ be a $p$-rough group.
\begin{itemize}
\item[(a)] If $|G'|=p$, then $G'\le Z(G)$; in particular, $G$ is nilpotent.
\item[(b)] If $G$ is not nilpotent, then $|G'|\ge2p+1$.
\end{itemize}
\end{lemma}

\begin{proof}
(a) The group $G/C_G(G')$ is isomorphic to a subgroup of $\Aut(G')$, which has order $p-1$. As every prime divisor of $|G|$ is at least $p$, it follows that $G=C_G(G')$, i.e., $G'\le Z(G)$. Then $G/Z(G)$ is abelian, and $G$ is nilpotent.

(b) As $G$ is solvable and not nilpotent, not all chief factors of $G$ are central, since otherwise a chief series of $G$ would be a central series. So let $H/K$ be a chief factor of $G$ with $[H,G]\not\le K$. Then $H/K$ is an elementary abelian $q$-group for some prime $q$, and it is an irreducible module over $\GF(q)$ for $G/C$, where $C=C_G(H/K)<G$. Since $[H,G]\le H\cap G'$, the normal subgroup $(H\cap G')K/K$ of $G/K$ is nontrivial and contained in $H/K$, so that $H=(H\cap G')K$ and $|H/K|=|H\cap G':K\cap G'|\le|G'|$. Let $R/C=O_q(G/C)$. As $R/C$ is a $q$-group, $C_{H/K}(R)\ne1$, and since $C_{H/K}(R)$ is normal in $G/K$, we get $C_{H/K}(R)=H/K$, i.e., $R=C$. As $G/C\ne1$, there is a prime $r\ne q$ dividing $|G/C|$, and we choose $x\in G$ such that $xC$ has order $r$. Then $r\ge p$, and $x$ acts nontrivially on $H/K$. By Maschke's theorem, there is an irreducible $\langle x\rangle$-submodule $U$ of $H/K$ on which $x$ acts nontrivially. Then $C_U(x)$ is an $\langle x\rangle$-submodule of $U$ different from $U$, so $C_U(x)=1$, and $\langle xC\rangle$ acts semiregularly on $U^\#$. Hence $r$ divides $|U|-1$. As $|U|$ and $r$ are odd, $|U|-1$ is an even multiple of $r$, and therefore $|G'|\ge|H/K|\ge|U|\ge2r+1\ge2p+1$.
\end{proof}

\begin{lemma}\label{abelian}
Let $A$ be an abelian group acting faithfully and completely reducibly on the finite vector space $V$, possibly of mixed characteristic.
\begin{itemize}
\item[(a)] If $W\le V$ is an $A$-submodule such that $W_A$ is homogeneous, then the subalgebra of $\End(W)$ generated by the image of $A$ is a field, and $C_A(w)=C_A(W)$ for all $w\in W^\#$.
\item[(b)] $A$ has a regular orbit on $V$.
\item[(c)] The set of all $v\in V$ with $C_A(v)=1$ spans $V$.
\end{itemize}
\end{lemma}

\begin{proof}
(a) Let $W$ be defined over the field $F$, and let $U$ be an irreducible constituent of $W_A$. As $W_A\cong U\oplus\cdots\oplus U$, the $F$-subalgebra $E$ of $\End_F(W)$ generated by the image of $A$ is isomorphic to the image of $FA$ in $\End_F(U)$. As $A$ is abelian, the latter lies in $\End_{FA}(U)$, which is a division ring by Schur's lemma. Thus $E$ is a finite commutative ring without zero divisors, i.e., a field. If $a\in A\setminus C_A(W)$, then $a-1$ is a nonzero element of $E$, hence invertible, and so $a$ fixes no vector in $W^\#$.

(b) and (c). Let $V_1,\ldots,V_n$ be the homogeneous components of $V_A$, and let $x=x_1+\cdots+x_n\in V$ with $x_j\in V_j$. For each $j$ choose $v_j\in V_j^\#$ with $v_j\ne-x_j$, unless $A$ acts trivially on $V_j$; this is possible, as $|V_j|>2$ if $A$ acts nontrivially on $V_j$. Put $v=v_1+\cdots+v_n$. Then $v_j+x_j\ne0$ unless $C_A(V_j)=A$, so (a) yields $C_A(v)=C_A(v+x)=\bigcap_jC_A(V_j)=1$. For $x=0$ this proves (b), and (c) follows from $x=(v+x)-v$.
\end{proof}

\begin{corollary}\label{collapse}
Let $G$ be a finite group of odd order and let $V$ be a finite faithful completely reducible $G$-module, possibly of mixed characteristic. Then $|G:G'|=M(G,V)$ if and only if $G$ is abelian.
\end{corollary}

\begin{proof}
If $G$ is nonabelian and $|G:G'|=M(G,V)$, then by Theorem~\ref{CK} we have $G=P\times A$, where $P$ is the Sylow $2$-subgroup of $G$ and $A$ is abelian. As $|G|$ is odd, $P=1$, and $G=A$ is abelian, a contradiction. Conversely, if $G$ is abelian, then $M(G,V)=|G|=|G:G'|$ by Lemma~\ref{abelian}(b).
\end{proof}

The following structural fact is well known.

\begin{lemma}\label{semilinear}
Let $G$ be a finite solvable group and let $V$ be a finite faithful irreducible $G$-module over the field $F$. Suppose that $T=F(G)$ is abelian and that $V_T$ is homogeneous; for instance, let $V$ be quasiprimitive and $F(G)$ cyclic. Then there is a finite field $E\supseteq F$ such that $V$ can be identified with $E$ in such a way that $G\le\Gamma(E)$ and $T=G\cap\Gamma_0(E)$. In particular, $T$ is cyclic.
\end{lemma}

\begin{proof} This follows immediately from \cite[Theorem 2.1 and Lemma 2.2]{MW93}.
\end{proof}

\begin{lemma}[Dixon \cite{Di67}]\label{Dixon}
Let $P$ be a permutation group of odd order on a set of $t\ge1$ elements. Then $|P|\le(\sqrt3)^{t-1}$.
\end{lemma}

The existence of regular orbits of nilpotent groups is governed by the following result of Yang; for coprime actions, results of this type go back to Hargraves \cite{Ha81}.

\begin{theorem}[{\cite[Theorem 1.1]{Ya11}}]\label{Yang}
Let $N$ be a finite nilpotent group and let $V$ be a finite faithful $N$-module over a field of characteristic $s$ such that $O_s(N)=1$. Suppose that $N$ does not involve $C_2\wr C_2$ and, in case $s=2$, that $N$ does not involve $C_r\wr C_r$ for any Mersenne prime $r$. Then $N$ has a regular orbit on $V$.
\end{theorem}

\begin{lemma}\label{nilreg}
Let $N$ be a finite nilpotent group of odd order and let $V$ be a finite faithful completely reducible $N$-module, possibly of mixed characteristic. Suppose that $|V|$ is odd, or that $N$ does not involve $C_r\wr C_r$ for any Mersenne prime $r$. Then $N$ has a regular orbit on $V$.
\end{lemma}

\begin{proof}
Let $p_1,\ldots,p_t$ be the distinct primes occurring as characteristics of the irreducible summands of $V$, and write $V=V_1\oplus\cdots\oplus V_t$, where $V_j$ is the sum of the irreducible summands of characteristic $p_j$. Put $N_0=N$ and $N_j=C_N(V_1\oplus\cdots\oplus V_j)$ for $1\le j\le t$. Then $N_j\nsg N$, $N_t=1$, and $N_{j-1}/N_j$ acts faithfully on $V_j$.

Fix $j\in\{1,\ldots,t\}$ and let $R/N_j=O_{p_j}(N_{j-1}/N_j)$. By Clifford's theorem, $V_j$ is a direct sum of irreducible $N_{j-1}$-modules $W$, each of them defined over some finite field of characteristic $p_j$. As $R$ acts on $W$ as a $p_j$-group and $|W|$ is a power of $p_j$, we have $C_W(R)\ne0$, and since $R\nsg N_{j-1}$, the space $C_W(R)$ is an $N_{j-1}$-submodule of $W$. So $C_W(R)=W$. Thus $R$ acts trivially on $V_j$, i.e., $R=N_j$. As $|N|$ is odd, $N_{j-1}/N_j$ does not involve $C_2\wr C_2$, and if $p_j=2$, then $|V|$ is even and hence $N_{j-1}/N_j$ does not involve $C_r\wr C_r$ for any Mersenne prime $r$. We may therefore apply Theorem~\ref{Yang} to the nilpotent group $N_{j-1}/N_j$ acting faithfully on $V_j$, regarded as a vector space over $\GF(p_j)$, and we obtain $v_j\in V_j$ with $C_{N_{j-1}}(v_j)=N_j$.

Put $v=v_1+\cdots+v_t$. As all $V_j$ are $N$-invariant, we have $C_N(v_1+\cdots+v_j)=C_N(v_1+\cdots+v_{j-1})\cap C_N(v_j)$ for $j\ge2$, and induction on $j$ yields
\[
C_N(v_1+\cdots+v_j)=N_{j-1}\cap C_N(v_j)=C_{N_{j-1}}(v_j)=N_j
\]
for $1\le j\le t$. For $j=t$ we get $C_N(v)=N_t=1$, as desired.
\end{proof}

The next three results will only be used for modules of odd order in Section~\ref{sec:oddV}.

\begin{lemma}\label{minus}
Let $G$ be a group of odd order acting linearly on a finite vector space $V$ over a field of odd characteristic. Then $-v\notin v^G$ for every $v\in V^\#$. Consequently, there exists a $G$-invariant subset $S\subseteq V^\#$ such that $V^\#=S\cup(-S)$ and $S\cap(-S)=\emptyset$.
\end{lemma}

\begin{proof}
Suppose that $vg=-v$ for some $v\in V^\#$ and $g\in G$. Then $vg^2=v$. As $g$ has odd order, $g$ is a power of $g^2$, and thus $vg=v$. Hence $v=-v$, i.e., $2v=0$, which is impossible as $v\ne0$ and the characteristic is odd.

For the second assertion, note that the map $u\mapsto-u$ commutes with the action of $G$ and hence permutes the $G$-orbits on $V^\#$; by what we have just shown, it maps no orbit to itself. Therefore the $G$-orbits on $V^\#$ fall into pairs $\{O,-O\}$ with $O\ne-O$, and we may take $S$ to be the union of one orbit from each pair.
\end{proof}

\begin{lemma}[{\cite[Lemma 2.1]{Es91}}]\label{Espuelas}
Let $G$ be a finite group of odd order and let $V$ be a finite faithful irreducible quasiprimitive $G$-module of odd characteristic. If $F(G)$ is not cyclic, then $G$ has at least two regular orbits on $V$.
\end{lemma}

\begin{theorem}[{\cite[Corollary 1]{Gl83}}]\label{Gluck}
Let $P$ be a permutation group of odd order on a finite set $\Omega$. Then there exists a subset $\Delta\subseteq\Omega$ whose setwise stabilizer in $P$ is trivial.
\end{theorem}

\section{Nilpotent groups}\label{sec:nilpotent}

\begin{proposition}\label{nilpotent}
Let $p$ be an odd prime, let $G$ be a nonabelian nilpotent $p$-rough group, and let $V$ be a finite faithful completely reducible $G$-module, possibly of mixed characteristic. Then $\rho(G,V)\ge p$, i.e., $|G:G'|\le M(G,V)/p$.
\end{proposition}

\begin{proof}
Put $M=M(G,V)$. We first assume that $V$ is an irreducible $G$-module over the finite field $F$. As $G$ is nilpotent and nonabelian, some Sylow $q$-subgroup $Q$ of $G$ is nonabelian, and $q\ge p$. Write $G=Q\times H$, where $H$ is the Hall $q'$-subgroup of $G$. By \cite[Lemma 10]{RT96}, there is a finite extension field $L$ of $F$ such that $V$ can be regarded as an absolutely irreducible $LG$-module without changing the permutation action of $G$ on $V$ (compare Step 2 of the proof of \cite[Theorem 3.1]{CK24b}).

Suppose first that $H\ne1$. Then by \cite[(3.16)]{As84} we have $V=X\otimes_LY$, where $X$ is a faithful irreducible $LQ$-module and $Y$ is a faithful irreducible $LH$-module. Let $x\in X^\#$ and $y\in Y^\#$, and let $u\in Q$ and $h\in H$ with $uh\in C_G(x\otimes y)$. By \cite[Lemma 3.3]{LS02} there are $\alpha,\beta\in L$ with $xu=\alpha x$, $yh=\beta y$ and $\alpha\beta=1$. As $\alpha^{o(u)}=1=\beta^{o(h)}$ and $o(u)$ and $o(h)$ are coprime, we get $\alpha=\beta=1$. Thus $C_G(x\otimes y)=C_Q(x)\times C_H(y)$ and $|(x\otimes y)^G|=|x^Q|\cdot|y^H|$. Now choose $x$ in an orbit of maximal size of $Q$ on $X$, and $y$ in an orbit of maximal size of $H$ on $Y$. By Theorem~\ref{KY} we have $|y^H|\ge|H:H'|$ and $|x^Q|\ge|Q:Q'|$, and by Corollary~\ref{collapse} the latter inequality is strict, as $Q$ is a nonabelian group of odd order. Since $|x^Q|$ and $|Q:Q'|$ are powers of $q$, we obtain $|x^Q|\ge q|Q:Q'|\ge p|Q:Q'|$, and therefore
\[
M\ge|(x\otimes y)^G|\ge p|Q:Q'|\cdot|H:H'|=p|G:G'|.
\]
If $H=1$, then $G=Q$, and the same argument with $X=V$ gives $M=M(Q,V)\ge q|Q:Q'|\ge p|G:G'|$.

Now let $V$ be arbitrary, and write $V=V_1\oplus\cdots\oplus V_t$ with irreducible $G$-modules $V_i$. As $G$ is nonabelian and isomorphic to a subgroup of the direct product of the groups $G/C_G(V_i)$, we may assume that $\bar G=G/C_G(V_1)$ is nonabelian. Put $C=C_G(V_1)$ and $W=V_2\oplus\cdots\oplus V_t$. By the irreducible case, applied to the nonabelian nilpotent $p$-rough group $\bar G$, we have $M(G,V_1)=M(\bar G,V_1)\ge p|\bar G:\bar G'|$. Moreover, $M(C,W)\ge|C:C'|$ by Theorem~\ref{KY}, as $W$ is a faithful completely reducible $C$-module by Clifford's theorem. By Lemmas~\ref{product} and~\ref{quot},
\[
M\ge M(G,V_1)\cdot M(C,W)\ge p|\bar G:\bar G'|\cdot|C:C'|\ge p|G:G'|,
\]
as wanted.
\end{proof}

\begin{proposition}\label{Gprime}
Let $p$ be an odd prime and let $G$ be a $p$-rough group with $|G'|=p$. Then $G$ is nilpotent, and $G$ has a regular orbit on every finite faithful completely reducible $G$-module $V$, possibly of mixed characteristic. In particular, $\rho(G,V)=p$ for all such $V$.
\end{proposition}

\begin{proof}
By Lemma~\ref{comm}(a), $G$ is nilpotent. Suppose that $G$ involves $C_r\wr C_r$ for some odd prime $r$, i.e., that there are subgroups $N\nsg U\le G$ with $U/N\cong C_r\wr C_r$. Then $r\ge p$. On the other hand, $(U/N)'=U'N/N$ is a quotient of $U'\le G'$, and hence $|(U/N)'|\le p$. But the derived subgroup of $C_r\wr C_r=B\rtimes\langle s\rangle$, where $B\cong C_r^r$ is the base group, is $[B,s]$, which has order $r^{r-1}\ge p^{p-1}>p$. This contradiction shows that $G$ does not involve $C_r\wr C_r$ for any odd prime $r$. So $G$ has a regular orbit on $V$ by Lemma~\ref{nilreg}, whence $M(G,V)=|G|$ and $\rho(G,V)=|G'|=p$.
\end{proof}

\section{Semilinear groups}\label{sec:semilinear}

In this section we treat the case that $F(G)$ is abelian and $V_{F(G)}$ is homogeneous, so that $G$ is a semilinear group by Lemma~\ref{semilinear}. We need two elementary lemmas.

\begin{lemma}\label{arith}
Let $A$ be a cyclic group of odd order, let $\alpha$ be an automorphism of $A$ of odd order $k$, and let $e=|[A,\alpha]|$. Then one of the following holds.
\begin{itemize}
\item[(a)] For every prime $\ell$ dividing $|A|$, $\alpha$ induces an automorphism of $\ell$-power order on the Sylow $\ell$-subgroup of $A$.
\item[(b)] $2k<e$.
\end{itemize}
\end{lemma}

\begin{proof}
Let $n=|A|$, let $a_0$ be a generator of $A$, and let $m$ be an integer with $a_0^\alpha=a_0^m$. Then $a^\alpha=a^m$ for all $a\in A$, and $k$ is the order of $m+n\Z$ in $(\Z/n\Z)^\times$. Let $h=\gcd(n,m-1)$. As $A$ is abelian, $a\mapsto a^{-1}a^\alpha=a^{m-1}$ is an endomorphism of $A$ with image $[A,\alpha]=\langle a_0^{m-1}\rangle$, so that $e=n/h$. Write $n=\prod_\ell\ell^{a_\ell}$ and $h=\prod_\ell\ell^{b_\ell}$ with $0\le b_\ell\le a_\ell$, where $\ell$ runs over the prime divisors of $n$. The residue class of $m$ lies in the subgroup
\[
U=\{u\in(\Z/n\Z)^\times\mid u\equiv1\pmod h\}
\]
of $(\Z/n\Z)^\times$, and so $k$ divides $|U|$. By the Chinese remainder theorem, $U\cong\prod_\ell U_\ell$ with $U_\ell=\{u\in(\Z/\ell^{a_\ell}\Z)^\times\mid u\equiv1\pmod{\ell^{b_\ell}}\}$, where $|U_\ell|=\ell^{a_\ell-b_\ell}$ if $b_\ell\ge1$, and $|U_\ell|=\ell^{a_\ell-1}(\ell-1)$ if $b_\ell=0$. Since $e=\prod_\ell\ell^{a_\ell-b_\ell}$, we obtain
\[
|U|=e\cdot\prod_{\ell\mid n,\ \ell\nmid h}\frac{\ell-1}{\ell}.
\]
If every prime divisor $\ell$ of $n$ divides $h$, then $m\equiv1\pmod\ell$ for all these $\ell$. As the kernel of the natural map $(\Z/\ell^{a_\ell}\Z)^\times\to(\Z/\ell\Z)^\times$ has order $\ell^{a_\ell-1}$, it follows that $\alpha$, which acts on the Sylow $\ell$-subgroup of $A$ by raising to the $m$-th power, induces an automorphism of $\ell$-power order on it; so (a) holds. Otherwise there is a prime divisor $\ell$ of $n$ not dividing $h$; then $|U|<e$, and $|U|$ is even because $\ell-1$ is even. As $k$ is odd and divides $|U|$, we get $k\le|U|/2<e/2$, which is (b).
\end{proof}

\begin{lemma}\label{nilcrit}
Let $G=A\langle g\rangle$ be a finite group, where $A$ is a cyclic normal subgroup of $G$. Suppose that for every prime $\ell$ dividing $|A|$, the element $g$ induces an automorphism of $\ell$-power order on the Sylow $\ell$-subgroup $A_\ell$ of $A$. Then $G$ is nilpotent.
\end{lemma}

\begin{proof}
Let $\ell$ be a prime divisor of $|A|$. Since $A$ is abelian, $A\le C_G(A_\ell)$, so $G/C_G(A_\ell)$ is generated by $gC_G(A_\ell)$ and is isomorphic to the group of automorphisms of $A_\ell$ generated by the one induced by $g$. Hence $G/C_G(A_\ell)$ is an $\ell$-group, and every $\ell'$-element of $G$ centralizes $A_\ell$. Now let $q$ be a prime and $S\in\Syl_q(G)$. Then $S$ centralizes $A_\ell$ for every prime $\ell\ne q$, i.e., $S$ centralizes the Hall $q'$-subgroup $A_{q'}$ of $A$. Moreover, $A_q$ is a normal $q$-subgroup of $G$, and so $A_q\le S$. Therefore $SA=SA_{q'}=S\times A_{q'}$. As $G/A$ is abelian, $SA\nsg G$, and as $S$ is the unique Sylow $q$-subgroup of $SA$, it follows that $S\nsg G$. Thus every Sylow subgroup of $G$ is normal, and $G$ is nilpotent.
\end{proof}

Next we record the properties of the exceptional groups $D_1$ and $D_2$ defined in the introduction.

\begin{lemma}\label{exc}
Let $i\in\{1,2\}$. Then $D_i'$ is cyclic of order $7$, and every $u\in U_i^\#$ satisfies $|C_{D_i}(u)|=3$. In particular, $M(D_i,U_i)=|D_i|/3$ and
\[
|D_i:D_i'|=\frac37M(D_i,U_i).
\]
Moreover, $D_i$ is generated by the stabilizers $C_{D_i}(u)$ with $u\in U_i^\#$, and $D_2\cong\Gamma(2^3)\times C_3$.
\end{lemma}

\begin{proof}
For $i=1$ we have $D_1=\Gamma_0(8)\rtimes\langle x\mapsto x^2\rangle\cong C_7\rtimes C_3$, where the automorphism $x\mapsto x^2$ acts fixed point freely on $\Gamma_0(8)\cong C_7$. Hence $D_1'=\Gamma_0(8)$ and $|D_1:D_1'|=3$. Since $\Gamma_0(8)$ is regular on $U_1^\#$ and $|D_1|=21$, every $u\in U_1^\#$ has a stabilizer of order $3$, which is a Sylow $3$-subgroup of $D_1$. As $D_1$ has seven Sylow $3$-subgroups, each of which fixes exactly one vector in $U_1^\#$ (the fixed field of $x\mapsto x^2$ being $\GF(2)$), all Sylow $3$-subgroups occur as stabilizers, and they generate $D_1$.

For $i=2$, the automorphism $\sigma_0\colon x\mapsto x^4$ of $\GF(64)$ has order $3$ and fixed field $\GF(4)$. Write $A_0=Z\times Q$ with $Z=\GF(4)^\times$ of order $3$ and $Q$ of order $7$. Then $\sigma_0$ centralizes $Z$ and acts on $Q\le\GF(8)^\times$ by $a\mapsto a^4$, i.e., fixed point freely. Hence $D_2=Z\times(Q\rtimes\langle\sigma_0\rangle)$, where $Q\rtimes\langle\sigma_0\rangle\cong\Gamma(2^3)$ (note that $\sigma_0$ restricts to a generator of $\Gal(\GF(8))$), and $D_2'=Q$ has order $7$, so $|D_2:D_2'|=9$. Moreover, the multiplication map $\GF(8)\otimes_{\GF(2)}\GF(4)\to\GF(64)$ is a ring isomorphism, under which $\sigma_0$ corresponds to $(x\mapsto x^4)\otimes\mathrm{id}$, multiplication by $a\in Q$ to $a\otimes\mathrm{id}$ and multiplication by $z\in Z$ to $\mathrm{id}\otimes z$; this yields the description of $(D_2,U_2)$ given in the introduction. As $\GF(64)^\times/A_0$ has order $3$, we have $x^3\in A_0$ and thus $x^{\sigma_0}=x\cdot x^3\in xA_0$ for all $x\in\GF(64)^\times$. Hence the three orbits $xA_0$ of $A_0$ on $U_2^\#$ are $\sigma_0$-invariant, and $D_2$ has exactly three orbits on $U_2^\#$, each of size $21=|D_2|/3$. For $x\in\GF(64)^\times$, the map $y\mapsto x^{-3}y^4$ lies in $D_2$ and fixes $x$, so it generates $C_{D_2}(x)$. Since $x\mapsto x^{-3}$ maps $\GF(64)^\times$ onto $A_0$, the stabilizers of the vectors in $U_2^\#$ are exactly the subgroups generated by the maps $y\mapsto ay^4$ with $a\in A_0$. These subgroups generate $D_2$: they contain $\sigma_0$, and for $a\in A_0$ the map $y\mapsto ay^4$ followed by $\sigma_0^{-1}\colon y\mapsto y^{16}$ is the multiplication $y\mapsto a^{16}y$, where $a^{16}$ runs through $A_0$ as $a$ does.

In both cases we have $M(D_i,U_i)=|D_i|/3$, and hence
\[
|D_i:D_i'|=\frac{|D_i|}{7}=\frac37M(D_i,U_i).
\]
\end{proof}

We now come to the main result of this section.

\begin{proposition}\label{semi}
Let $p$ be an odd prime, let $G$ be a non-nilpotent $p$-rough group, and let $V$ be a finite faithful irreducible $G$-module. Suppose that $F(G)$ is abelian and that $V_{F(G)}$ is homogeneous.
\begin{itemize}
\item[(a)] If $V$ has odd characteristic, then there exists $v\in V$ with $|C_G(v)|\le\frac3{13}|G'|$, and if $p\ge5$, then there exists $v\in V$ with $|C_G(v)|<\frac1{2p}|G'|$. Moreover, if $|C_G(v)|\ge\frac3{13}|G'|$ for all $v\in V$, then $|V|=27$ and $|G|=39$.
\item[(b)] If $V$ has characteristic $2$, then $\rho(G,V)\ge p$, unless $(G,V)$ is isomorphic to $(D_1,U_1)$ or to $(D_2,U_2)$, in which case $p=3$ and $\rho(G,V)=7/3$.
\end{itemize}
In particular, $\rho(G,V)\ge p$ unless $(G,V)$ is isomorphic to $(D_1,U_1)$ or to $(D_2,U_2)$.
\end{proposition}

\begin{proof}
By Lemma~\ref{semilinear} we may assume that $V=E$ for a finite field $E$, that $G\le\Gamma(E)$, and that $A:=F(G)=G\cap\Gamma_0(E)$. We identify $A$ with a subgroup of $E^\times$, an element $a\in A$ acting as $x\mapsto ax$. The map $\Gamma(E)\to\Gal(E)$ sending $x\mapsto bx^\tau$ to $\tau$ is a homomorphism with kernel $\Gamma_0(E)$, and so $G/A$ is isomorphic to a subgroup of the cyclic group $\Gal(E)$. Choose $g\in G$ with $G=A\langle g\rangle$, and let $c\in E^\times$ and $\sigma\in\Gal(E)$ be such that $g$ acts as $x\mapsto cx^\sigma$. Then $\sigma$ has order $k:=|G:A|$. Let $E_0$ be the fixed field of $\sigma$ and $r=|E_0|$, so that $|E|=r^k$. As $G$ is $p$-rough and not abelian, $k>1$ and every prime divisor of $k$ is at least $p$; in particular, $k$ is odd and $k\ge p$.

\smallskip\noindent\emph{(1) For $v\in V^\#$ we have $C_G(v)\cap A=1$, and $|C_G(v)|$ divides $k$.} Indeed, a nontrivial element of $\Gamma_0(E)$ has no nonzero fixed point, so $C_G(v)\cap A=1$ and $C_G(v)\cong C_G(v)A/A\le G/A$.

\smallskip\noindent\emph{(2) Let $e=|G'|$ and $d=|C_A(g)|$. Then $G'=[A,g]$, $|A|=de$, and $d\le r-1$; if $V$ has odd characteristic, then $d\le(r-1)/2$.} As $A$ is abelian, $a\mapsto a^{-1}a^g$ is a homomorphism from $A$ onto $[A,g]$ with kernel $C_A(g)$, so that $[A,g]$ is a subgroup of $A$ of order $|A|/d$. It is normalized by $g$, since $(a^{-1}a^g)^g=(a^g)^{-1}(a^g)^g$, and hence $[A,g]\nsg G$. Modulo $[A,g]$ the element $g$ centralizes $A$, so $G/[A,g]$ is abelian. Thus $G'\le[A,g]\le G'$, whence $G'=[A,g]$ and $|A|=de$. For $a\in A$ and $x\in E$, if we write the action of $\Gamma(E)$ on $(E,+)$ from the right, we have $x(ag)=ca^\sigma x^\sigma$ and $x(ga)=acx^\sigma$, so $a$ commutes with $g$ if and only if $a\in E_0$. Hence $C_A(g)\le E_0^\times$, and so $d$ divides $r-1$. If the characteristic is odd, then $r-1$ is even while $d$ is odd, and so $d\le(r-1)/2$.

\smallskip\noindent\emph{(3) $2k<e$.} Since $G$ is solvable, $C_G(A)=C_G(F(G))\le F(G)=A$ by \cite[III.4.2]{Hu67}. So $G/A$ acts faithfully on $A$, and the automorphism $\alpha$ of $A$ induced by $g$ has order $k$. By (2), $|[A,\alpha]|=|[A,g]|=e$. If $\alpha$ induced an automorphism of $\ell$-power order on every Sylow $\ell$-subgroup of $A$, then $G$ would be nilpotent by Lemma~\ref{nilcrit}. Hence Lemma~\ref{arith} yields $2k<e$.

\smallskip\noindent\emph{(4) If there exists $v\in V^\#$ with $|C_G(v)|<k$, then $|C_G(v)|<\frac1{2p}e\le\frac3{13}e$, and $\rho(G,V)>2p$. In particular, all assertions hold in this case.} Indeed, by (1), $|C_G(v)|$ then is a proper divisor of $k$, and as every prime divisor of $k$ is at least $p$, we get $|C_G(v)|\le k/p<e/(2p)$ by (3). Moreover, $1/(2p)\le1/6<3/13$, and $\rho(G,V)\ge|G:C_G(v)|/|G:G'|=e/|C_G(v)|>2p$.

\smallskip\noindent\emph{(5) If $|C_G(v)|=k$ for all $v\in V^\#$, then $\rho(G,V)=e/k$ and $r^k-1\le de(r-1)$.} The first assertion is clear, as then $M(G,V)=|G|/k$. By (1), $C_G(v)$ is a complement of $A$ in $G$ for every $v\in V^\#$. Let $\mathcal X$ be the set of all complements of $A$ in $G$. For $X\in\mathcal X$, the natural map $X\to G/A$ is an isomorphism, so $X$ contains a unique element $x_X$ with $x_XA=gA$, and $X=\langle x_X\rangle$. Hence $X\mapsto x_X$ is an injective map from $\mathcal X$ into $gA$, and $|\mathcal X|\le|A|$. Moreover, $x_X=ga$ for some $a\in A$, so that $x_X$ acts as $x\mapsto c'x^\sigma$ for some $c'\in E^\times$. If $y_0\in C_E(X)^\#$ and $y\in C_E(X)$, then $c'y_0^\sigma=y_0$ and $c'y^\sigma=y$, whence $(y/y_0)^\sigma=y/y_0$, i.e., $y\in y_0E_0$. Therefore $|C_E(X)^\#|\in\{0,r-1\}$. If $X,Y\in\mathcal X$ are distinct and $y\in C_E(X)\cap C_E(Y)$ is nonzero, then $\langle X,Y\rangle\le C_G(y)$, which is impossible as $|\langle X,Y\rangle|>|X|=k=|C_G(y)|$. Since every $y\in E^\#$ lies in $C_E(C_G(y))$ and $C_G(y)\in\mathcal X$, we conclude that $E^\#$ is the disjoint union of the sets $C_E(X)^\#$ with $X\in\mathcal X$. Consequently $r^k-1\le|\mathcal X|(r-1)\le|A|(r-1)=de(r-1)$.

\smallskip\noindent\emph{(6) Proof of (a).} By (4) we may assume that $|C_G(v)|=k$ for all $v\in V^\#$. Now $r\ge3$ is odd, and (5) and (2) yield
\[
\frac ek\ge\varphi(r,k):=\frac{2(r^k-1)}{k(r-1)^2}=\frac{2(1+r+\cdots+r^{k-1})}{k(r-1)}.
\]
If $k=3$, then
\[
\varphi(r,3)-\frac{13}{3}=\frac{2(r^2+r+1)-13(r-1)}{3(r-1)}=\frac{(2r-5)(r-3)}{3(r-1)}\ge0,
\]
with equality if and only if $r=3$. If $k\ge5$, then
\[
\varphi(r,k)>\frac{2r^{k-1}}{k(r-1)}>\frac{2r^{k-2}}{k}\ge\frac{2\cdot3^{k-2}}{k}\ge\frac{54}{5}>\frac{13}{3},
\]
where we have used that $3^{x-2}/x$ is increasing for $x\ge2$. If moreover $p\ge5$, then $k\ge p$, and the same estimate yields $\varphi(r,k)>2\cdot3^{p-2}/p>2p$, since $3^{x-2}>x^2$ for all integers $x\ge5$. In any case $e/k\ge13/3$, and if $p\ge5$ then $e/k>2p$ (note that $k\ge5$ if $p\ge5$). Since $|C_G(v)|=k$ for all $v\in V^\#$, the first two assertions of (a) follow. If moreover $|C_G(v)|\ge3e/13$ for all $v\in V$, then $e/k=13/3=\varphi(r,k)$, which forces $k=r=3$ and $e=13$. Then $|V|=|E|=r^k=27$, $d\le(r-1)/2=1$, and $|G|=k|A|=kde=39$.

\smallskip\noindent\emph{(7) Proof of (b).} Let $V$ have characteristic $2$. By (4) we may assume that $|C_G(v)|=k$ for all $v\in V^\#$. Then $r\ge2$ is a power of $2$, $d\le r-1$, and (5) yields
\[
\rho(G,V)=\frac ek\ge\psi(r,k):=\frac{r^k-1}{k(r-1)^2}=\frac{1+r+\cdots+r^{k-1}}{k(r-1)}.
\]
Suppose first that $k\ge5$, and put $m=\max\{5,p\}$, so that $k\ge m$. If $r=2$, then $\psi(r,k)=(2^k-1)/k$, and if $r\ge4$, then $\psi(r,k)>r^{k-1}/(k(r-1))>r^{k-2}/k\ge4^{k-2}/k$. As $(2^x-1)/x$ and $4^{x-2}/x$ are increasing for $x\ge1$, we get $\rho(G,V)\ge\min\{(2^m-1)/m,\,4^{m-2}/m\}$. For $p=3$ this is $31/5>3$, and for $p\ge5$ it exceeds $p$, since $2^x-1>x^2$ and $4^{x-2}>x^2$ for all integers $x\ge5$. So $\rho(G,V)>p$ if $k\ge5$.

Now let $k=3$. Then $p=3$, as $p\le k$. Here
\[
\psi(r,3)-3=\frac{r^2-8r+10}{3(r-1)},
\]
which is positive for $r\ge8$. Hence $\rho(G,V)\ge3$ unless $r\in\{2,4\}$. In this remaining case suppose that $\rho(G,V)<3$. Then $e<3k=9$, and as $e>2k=6$ by (3) and $e$ is odd, we get $e=7$ and $\rho(G,V)=7/3$. By (5) we then have $r^3-1\le7d(r-1)$, i.e., $d\ge(r^2+r+1)/7$, and together with $d\le r-1$ this yields $d=1$ if $r=2$, and $d=3$ if $r=4$.

If $r=2$, then $|E|=8$ and $|G|=k|A|=kde=21=|\Gamma(E)|$, so $G=\Gamma(E)$ and $(G,V)$ is isomorphic to $(D_1,U_1)$.

If $r=4$, then $|E|=64$, $|A|=21$, so $A=A_0$ is the subgroup of order $21$ of $E^\times$, and $\sigma\in\{\sigma_0,\sigma_0^2\}$, as these are the automorphisms of order $3$ of $E$. As in the proof of Lemma~\ref{exc}, we have $x^\sigma\in xA$ for all $x\in E^\times$, so that $g$ maps the orbit $xA$ of $A$ onto $cxA$. If $c\notin A$, then $g$ permutes the three orbits of $A$ on $E^\#$ transitively, so $G$ is transitive on $E^\#$ and, as $|G|=63=|E^\#|$, has a regular orbit, contradicting $|C_G(v)|=k$ for all $v\in E^\#$. Hence $c\in A$, and so $G=A\langle\sigma\rangle=A_0\langle\sigma_0\rangle=D_2$ and $(G,V)=(D_2,U_2)$.

Conversely, in both exceptional cases $p=3$ and $\rho(G,V)=7/3$ by Lemma~\ref{exc}. This proves (b).

Finally, if $V$ has odd characteristic, then (a) yields $\rho(G,V)\ge13/3>3$ if $p=3$, and $\rho(G,V)>2p$ if $p\ge5$. Together with (b), this proves the last assertion.
\end{proof}

\section{Invariant regular orbits of the Fitting subgroup}\label{sec:invariant}

We now treat the case that every regular orbit of the Fitting subgroup is invariant under $G$, but the restriction of $V$ to $F(G)$ is not homogeneous. We need the following computation in finite fields.

\begin{lemma}\label{galois}
Let $p$ be an odd prime, let $E$ be a finite field of characteristic $2$, and let $\sigma$ be a nontrivial automorphism of $E$ whose order $o$ has only prime divisors $\ge p$. Let $\sigma$ act on the cyclic group $E^\times$. Then
\[
|[[E^\times,\sigma],\sigma]|\ge\gamma_p\,o,\qquad\text{where}\quad\gamma_p=\frac{2^p-1}{p}.
\]
\end{lemma}

\begin{proof}
Let $E_0$ be the fixed field of $\sigma$ and $r=|E_0|\ge2$, so that $|E|=r^o$ and $\langle\sigma\rangle=\Gal(E/E_0)$ is generated by $\varphi\colon x\mapsto x^r$. For every automorphism $\beta$ of an abelian group $C$ and every $i\ge1$ we have $c^{-1}c^{\beta^i}=\prod_{j=0}^{i-1}(c^{\beta^j})^{-1}(c^{\beta^j})^{\beta}$, so that $[C,\beta^i]\le[C,\beta]$; as $\sigma$ and $\varphi$ generate the same group, it follows that $[E^\times,\sigma]=[E^\times,\varphi]$ and $[[E^\times,\sigma],\sigma]=[[E^\times,\varphi],\varphi]$. Now put $C=E^\times$, a cyclic group of order $r^o-1$. Then $[C,\varphi]=\{c^{r-1}\mid c\in C\}$ and $[[C,\varphi],\varphi]=\{c^{(r-1)^2}\mid c\in C\}$, which has order $(r^o-1)/\gcd(r^o-1,(r-1)^2)$. Let $s=(r^o-1)/(r-1)=1+r+\cdots+r^{o-1}$. Then $\gcd(r^o-1,(r-1)^2)=(r-1)\gcd(s,r-1)$, and as $s\equiv o\pmod{r-1}$, we obtain
\[
|[[C,\varphi],\varphi]|=\frac{s}{\gcd(o,r-1)}.
\]
Note that $o\ge p$. If $r=2$, this equals $2^o-1\ge\gamma_p\,o$, as $(2^x-1)/x$ is increasing for $x\ge1$. If $r\ge4$, then $s\ge(4^o-1)/3$ and $\gcd(o,r-1)\le o$, so that
\[
\frac{s}{\gcd(o,r-1)}\ge\frac{4^o-1}{3o^2}\cdot o\ge\frac{4^p-1}{3p^2}\cdot o\ge\gamma_p\,o,
\]
where we have used that $(4^x-1)/x^2$ is increasing for $x\ge1$ and that $(4^p-1)/(3p^2)=\gamma_p(2^p+1)/(3p)\ge\gamma_p$, as $2^p+1\ge3p$.
\end{proof}

\begin{proposition}\label{invariant}
Let $p$ be an odd prime, let $G$ be a $p$-rough group, and let $V$ be a finite faithful irreducible $G$-module. Suppose that $F=F(G)$ is abelian, that $V_F$ is not homogeneous, and that $FC_G(v)=G$ for every $v\in V$ with $C_F(v)=1$. Then $V$ has characteristic $2$, and $\rho(G,V)\ge p$.
\end{proposition}

\begin{proof}
Let $V_1,\ldots,V_t$ be the homogeneous components of $V_F$, where $t\ge2$. By Clifford's theorem, $G$ permutes $V_1,\ldots,V_t$ transitively. Let $K$ be the kernel of this permutation action and $H=N_G(V_1)$. As $t=|G:H|$ divides $|G|$ and $t>1$, we have $t\ge p$. By Lemma~\ref{abelian}(a), for each $j$ we have $C_F(v_j)=C_F(V_j)$ for all $v_j\in V_j^\#$. Consequently, a vector $v=v_1+\cdots+v_t$ with $v_j\in V_j$ satisfies $C_F(v)=1$ if and only if $\bigcap C_F(V_j)=1$, the intersection taken over all $j$ with $v_j\ne0$. In particular, every vector in
\[
R_0=\{v_1+\cdots+v_t\mid v_j\in V_j^\#\ \text{for all}\ j\}
\]
has trivial stabilizer in $F$. Moreover, if $C_F(v)=1$, then $FC_G(v)=G$ by hypothesis, so $C_G(v)$ maps onto $G/F$ and hence onto $G/K$, i.e.,
\begin{equation}\label{trans}\tag{$\ast$}
\text{$C_G(v)$ permutes $V_1,\ldots,V_t$ transitively whenever $C_F(v)=1$.}
\end{equation}

\smallskip\noindent\emph{(1) For every $j$, the subgroup $B_j=C_F(\bigoplus_{i\ne j}V_i)$ is nontrivial.} Suppose that $B_j=1$ for some $j$, and let $v=\sum_{i\ne j}v_i$ with $v_i\in V_i^\#$ for all $i\ne j$. Then $C_F(v)=B_j=1$, but $C_G(v)$ stabilizes the set $\{V_i\mid i\ne j\}$ of components in which $v$ has a nonzero component, contradicting \eqref{trans}.

\smallskip\noindent\emph{(2) Fix $v_0=v_{0,1}+\cdots+v_{0,t}\in R_0$ with $v_{0,j}\in V_j^\#$, and put $X=C_G(v_0)$ and $X_1=X\cap H$. Then $G=F\rtimes X$, $H=F\rtimes X_1$, and $X_1$ fixes $v_{0,1}$.} Indeed, $X\cap F=C_F(v_0)=1$ and $FX=G$ by hypothesis. As $F\le K\le H$, we get $H=F(X\cap H)=FX_1$. Finally, $X_1$ fixes $v_0$ and leaves $V_1$ invariant, so it fixes $v_{0,1}$.

\smallskip\noindent\emph{(3) $H$ acts transitively on $V_1^\#$.} Fix $g_0\in G$ with $V_1g_0=V_2$. Let $v_1'\in V_1^\#$ and put $v=v_1'+v_{0,2}+\cdots+v_{0,t}\in R_0$. By \eqref{trans} there is $y\in C_G(v)$ with $V_1y=V_2$, and comparing the components in $V_2$ of $vy=v$, we obtain $v_1'y=v_{0,2}$. Now $y=hg_0$ for some $h\in H$, and hence $v_1'=v_{0,2}g_0^{-1}h^{-1}$ lies in the $H$-orbit of $v_{0,2}g_0^{-1}\in V_1^\#$. As $v_1'$ was arbitrary, the claim follows.

\smallskip\noindent\emph{(4) $V$ has characteristic $2$, and for each $j$, the image $F_j$ of $F$ in $\GL(V_j)$ acts regularly on $V_j^\#$. Moreover, if $E_j$ denotes the subalgebra of $\End(V_j)$ generated by $F_j$, then $E_j$ is a field, $V_j$ is a one-dimensional $E_j$-space, and $F_j=E_j^\times$.} Let $\bar H=H/C_H(V_1)$, and let $\bar F$ and $\bar X_1$ be the images of $F$ and $X_1$ in $\bar H$. By (3), $|V_1|-1=|\bar H:C_{\bar H}(v_{0,1})|$. By (2), $\bar X_1\le C_{\bar H}(v_{0,1})$ and $\bar H=\bar F\bar X_1$, so that $|\bar H:C_{\bar H}(v_{0,1})|\le|\bar H:\bar X_1|\le|\bar F|$. On the other hand, $\bar F=F_1$ acts fixed point freely on $V_1^\#$ by Lemma~\ref{abelian}(a), so $|F_1|\le|V_1|-1$. Hence $|F_1|=|V_1|-1$, and $F_1$ acts regularly on $V_1^\#$. As $|F_1|$ is odd, $|V_1|$ is even, so $V$ has characteristic $2$. By Lemma~\ref{abelian}(a), $E_1$ is a field, and $V_1$ is a vector space over $E_1$. As $E_1^\times\supseteq F_1$, we have $|E_1|\ge|V_1|$, and as $V_1\ne0$ is an $E_1$-space, $|V_1|\ge|E_1|$. Hence $V_1$ is one-dimensional over $E_1$ and $F_1=E_1^\times$. Since $G$ permutes the $V_j$ transitively and normalizes $F$, the same holds for every $j$.

\smallskip\noindent\emph{(5) Let $T\le\GL(V)$ be the group of all elements that leave every $V_j$ invariant and act on $V_j$ as an element of $E_j^\times$, for all $j$. Then $F\le T$, $X$ normalizes $T$, and $[T,X]\le F$.} By (4), $F\le T$. Let $x\in X$, and let $V_ix=V_j$. As $x$ normalizes $F$, conjugation by $x$ maps $F_i$ onto $F_j$ and hence $E_i$ onto $E_j$; thus $x^{-1}Tx=T$. Since each $E_j^\times$ acts regularly on $V_j^\#$, the group $T$ acts regularly on $R_0$, and $R_0=v_0T$. Let $\tau\in T$ and $v=v_0\tau$. As $C_F(v)=1$, we have $G=C_G(v)F$, so every $x\in X$ can be written as $x=cf$ with $c\in C_G(v)$ and $f\in F$; then $vx=vf$. On the other hand, $vx=v_0x(x^{-1}\tau x)=v_0\tau^x$, as $x$ fixes $v_0$. So $v_0\tau^x=v_0\tau f$, and the regularity of $T$ on $R_0$ yields $\tau^x=\tau f$, i.e., $\tau^{-1}\tau^x\in F$. This proves $[T,X]\le F$.

\smallskip\noindent\emph{(6) Let $Y=X\cap K$. We identify $V_j$ with $E_j$ via $v_{0,j}\lambda\mapsto\lambda$. Then every $y\in Y$ acts on each $V_j$ as a field automorphism $\sigma_j(y)$ of $E_j$. Let $Y_j=\{\sigma_j(y)\mid y\in Y\}\le\Gal(E_j)$. Then all $Y_j$ have the same order $o$, and $|Y|\le o^t$. Moreover, $|X|=|G:K|\cdot|Y|$ and $|G:K|\le(\sqrt3)^{t-1}$.} Let $y\in Y$. Then $y$ leaves every $V_j$ invariant, fixes $v_{0,j}$, and normalizes $F_j$, hence $E_j$. So $\lambda\mapsto y^{-1}\lambda y$ is a field automorphism $\sigma_j(y)$ of $E_j$, and $(v_{0,j}\lambda)y=(v_{0,j}y)\lambda^{\sigma_j(y)}=v_{0,j}\lambda^{\sigma_j(y)}$ for $\lambda\in E_j$, which proves the first assertion. As $Y\nsg X$ and $X$ permutes the $V_j$ transitively, conjugation by elements of $X$ shows that all $Y_j$ have the same order $o$. Since $Y$ acts faithfully on $V$, the map $y\mapsto(\sigma_j(y))_j$ is injective, so that $|Y|\le\prod_j|Y_j|=o^t$. Next, $XK\supseteq XF=G$, so $X/Y\cong XK/K=G/K$ and $|X|=|G:K|\cdot|Y|$. Finally, $G/K$ acts faithfully as a permutation group of odd order on $\{V_1,\ldots,V_t\}$, so $|G:K|\le(\sqrt3)^{t-1}$ by Lemma~\ref{Dixon}.

\smallskip\noindent\emph{(7) $\rho(G,V)\ge|G'|/|X|$.} Indeed, $|v_0^G|=|G:X|=|F|=|G|/|X|$, and $|G:G'|=|G|/|G'|$.

\smallskip\noindent\emph{(8) If $Y=1$, then $\rho(G,V)\ge p$.} Let $B=B_1\cdots B_t$. As $B_j$ acts trivially on $V_i$ for $i\ne j$ and faithfully on $V_j$, we have $B=B_1\times\cdots\times B_t$. Conjugation by $X$ permutes the $B_j$ transitively. If $b\in B_j$ and $x\in X$ satisfies $V_1x=V_j$, then $c=xbx^{-1}\in B_1$ and $b=c^x=c\cdot c^{-1}c^x\in B_1[B,X]$. Hence $B=B_1[B,X]$ and, as $|B_1|\ge p$ by (1),
\[
|G'|\ge|[B,X]|\ge|B|/|B_1|=|B_1|^{t-1}\ge p^{t-1}.
\]
Since $|X|=|G:K|\le(\sqrt3)^{t-1}$ by (6), (7) yields $\rho(G,V)\ge(p/\sqrt3)^{t-1}\ge(p/\sqrt3)^{p-1}$, as $t\ge p$. For $p=3$ the right hand side equals $3$, and for $p\ge5$ it is at least $p^4/9>p$.

\smallskip\noindent\emph{(9) If $Y\ne1$, then $\rho(G,V)>p$.} Now $o>1$, and as $Y_j$ is a homomorphic image of $Y$, every prime divisor of $o$ is at least $p$. For $\tau\in T$ and $y\in Y$, the element $y^{-1}\tau y$ acts on $V_j=E_j$ as multiplication by $\tau_j^{\sigma_j(y)}$, where $\tau_j\in E_j^\times$ denotes the element by which $\tau$ acts on $V_j$. Thus $Y$ acts on $T\cong\prod_jE_j^\times$ coordinatewise, and considering elements of $T$ with only one coordinate different from $1$, we see that $[T,Y]=\prod_j[E_j^\times,Y_j]$ and $[[T,Y],Y]=\prod_j[[E_j^\times,Y_j],Y_j]$. Each $Y_j$ is cyclic, being a subgroup of the cyclic group $\Gal(E_j)$, say $Y_j=\langle\sigma\rangle$. Then $[E_j^\times,Y_j]=[E_j^\times,\sigma]$ and $[[E_j^\times,Y_j],Y_j]=[[E_j^\times,\sigma],\sigma]$ (see the proof of Lemma~\ref{galois}), and so Lemma~\ref{galois} gives $|[[E_j^\times,Y_j],Y_j]|\ge\gamma_p\,o$. By (5) we have $[T,Y]\le[T,X]\le F$, hence $[[T,Y],Y]\le[F,Y]\le G'$, and therefore $|G'|\ge(\gamma_p\,o)^t$. By (6), $|X|\le(\sqrt3)^{t-1}o^t$, and so by (7)
\[
\rho(G,V)\ge\frac{\gamma_p^t}{(\sqrt3)^{t-1}}=\gamma_p\Big(\frac{\gamma_p}{\sqrt3}\Big)^{t-1}\ge\gamma_p\Big(\frac{\gamma_p}{\sqrt3}\Big)^{p-1},
\]
as $t\ge p$ and $\gamma_p\ge\gamma_3=7/3>\sqrt3$. For $p=3$ the right hand side equals $343/81>3$, and for $p\ge5$ it is at least $\gamma_p>p$, since $2^p-1>p^2$. This completes the proof.
\end{proof}

\section{Proof of Theorem~\ref{thmA} and Corollary~\ref{corB}}\label{sec:proofA}

We first treat irreducible modules.

\begin{theorem}\label{irred}
Let $p$ be an odd prime, let $G$ be a non-nilpotent $p$-rough group, and let $V$ be a finite faithful irreducible $G$-module. Then $\rho(G,V)\ge p$, unless $(G,V)$ is isomorphic to $(D_1,U_1)$ or to $(D_2,U_2)$, in which case $p=3$ and $\rho(G,V)=7/3$.
\end{theorem}

\begin{proof}
Let $F=F(G)$. By Theorem~\ref{KYF} we have $|G:G'|\le|F:F'|$, and $V$ is a faithful completely reducible $F$-module by Clifford's theorem.

Suppose first that $F$ is nonabelian. Then Proposition~\ref{nilpotent}, applied to the $p$-rough group $F$, yields
\[
M(G,V)\ge M(F,V)\ge p|F:F'|\ge p|G:G'|.
\]

So let $F$ be abelian. Then $|G:G'|\le|F|$, and by Lemma~\ref{abelian}(b) there exists $v\in V$ with $C_F(v)=1$. If $FC_G(v)\ne G$ for some such $v$, then $|G:FC_G(v)|\ge p$, as $G$ is $p$-rough, and
\begin{align*}
|v^G|&=|G:FC_G(v)|\cdot|FC_G(v):C_G(v)|\\
&=|G:FC_G(v)|\cdot|F|\ge p|F|\ge p|G:G'|.
\end{align*}
Thus we may assume that $FC_G(v)=G$ for every $v\in V$ with $C_F(v)=1$. If $V_F$ is not homogeneous, then Proposition~\ref{invariant} yields $\rho(G,V)\ge p$. If $V_F$ is homogeneous, then Proposition~\ref{semi} applies, and the proof is complete.
\end{proof}

Next we deal with completely reducible modules.

\begin{theorem}\label{cr}
Let $p$ be an odd prime, let $G$ be a non-nilpotent $p$-rough group, and let $V$ be a finite faithful completely reducible $G$-module, possibly of mixed characteristic. Then either $\rho(G,V)\ge p$, or $p=3$ and $\rho(G,V)=7/3$. Moreover, $\rho(G,V)=7/3$ if and only if $(G,V)$ is exceptional.
\end{theorem}

\begin{proof}
If $(G,V)$ is exceptional, say $G=D\times A$ and $V=U\oplus W$ as in the introduction, then $C_G(u+w)=C_D(u)\times C_A(w)$ for $u\in U$ and $w\in W$. By Lemma~\ref{exc} and Lemma~\ref{abelian}(b), $M(G,V)=\frac13|D|\cdot|A|$. As $G'=D'$ has order $7$, we get $|G:G'|=\frac17|D|\cdot|A|$, and so $\rho(G,V)=7/3$.

We now prove the remaining assertions by induction on $|G|$. Write $V=V_1\oplus\cdots\oplus V_t$ with irreducible $G$-modules $V_i$. Since $G$ is not nilpotent and isomorphic to a subgroup of the direct product of the groups $G/C_G(V_i)$, some $G/C_G(V_i)$ is not nilpotent, and we may assume that this holds for $i=1$. Put $C=C_G(V_1)$, $\bar G=G/C$ and $W=V_2\oplus\cdots\oplus V_t$. Then $C$ acts faithfully on $W$, and $W$ is a completely reducible $C$-module by Clifford's theorem. By Lemma~\ref{product},
\[
\rho(G,V)\ge\rho(\bar G,V_1)\cdot\mu\cdot\iota,
\]
where
\[
\mu=\frac{M(C,W)}{|C:C'|}\qquad\text{and}\qquad\iota=\frac{|\bar G:\bar G'|\cdot|C:C'|}{|G:G'|}.
\]
Here $\iota$ is an integer dividing $|G|$ by Lemma~\ref{quot}, so either $\iota=1$ or $\iota\ge p$. Moreover, $\mu\ge1$ by Theorem~\ref{KY}. If $C$ is abelian, then $\mu=1$ by Lemma~\ref{abelian}(b); if $C$ is nilpotent and nonabelian, then $\mu\ge p$ by Proposition~\ref{nilpotent}; and if $C$ is not nilpotent, then $\mu\ge7/3$ by induction, as $C<G$ is $p$-rough. Finally, by Theorem~\ref{irred} applied to the non-nilpotent $p$-rough group $\bar G$ acting faithfully and irreducibly on $V_1$, we have $\rho(\bar G,V_1)\ge p$, or $p=3$, $(\bar G,V_1)$ is isomorphic to $(D_1,U_1)$ or $(D_2,U_2)$, and $\rho(\bar G,V_1)=7/3$.

If $\rho(\bar G,V_1)\ge p$, then $\rho(G,V)\ge p$. So we may assume that $p=3$ and that $(\bar G,V_1)$ is isomorphic to $(D_i,U_i)$ for some $i$. If $\mu>1$ or $\iota>1$, then $\rho(G,V)\ge\frac73\min\{\frac73,3\}=\frac{49}9>3$. So we may also assume that $C$ is abelian and that $|G:G'|=|\bar G:\bar G'|\cdot|C|$. Then Lemma~\ref{product} and Lemma~\ref{exc} give $M(G,V)\ge M(\bar G,V_1)\cdot|C|=|G|/3$. As $M(G,V)=|G:C_G(v)|$ for some $v\in V$ with $|C_G(v)|$ odd, either $M(G,V)=|G|$, and then $\rho(G,V)=|G'|\ge7$ by Lemma~\ref{comm}(b), or $M(G,V)=|G|/3$, and then $\rho(G,V)=|G'|/3$. Now $|G'|=|G|/(|\bar G:\bar G'|\cdot|C|)=|\bar G'|=7$ by Lemma~\ref{exc}, so in the latter case $\rho(G,V)=7/3$. This proves that $\rho(G,V)\ge p$, or $p=3$ and $\rho(G,V)=7/3$.

Now suppose that $\rho(G,V)=7/3$. Then $p=3$, and by the preceding paragraph, $(\bar G,V_1)$ is isomorphic to some $(D_i,U_i)$, $C$ is abelian, $|G:G'|=|\bar G:\bar G'|\cdot|C|$, $|G'|=7$, and $M(G,V)=|G|/3$. By Lemma~\ref{quot}, $|G:G'|=|\bar G:\bar G'|\cdot|C:C\cap G'|$, so $C\cap G'=1$. Put $Q=G'$. We show that $(G,V)$ is exceptional in three steps.

\smallskip\noindent\emph{(a) For $v_1\in V_1^\#$ we have $|v_1^G|=|\bar G|/3$ and $|C_G(v_1):C|=3$.} This is immediate from Lemma~\ref{exc}, as $C_G(v_1)/C=C_{\bar G}(v_1)$.

\smallskip\noindent\emph{(b) $Q$ acts trivially on $W$.} Let $w\in W$ with $C_C(w)=1$, and let $v_1\in V_1^\#$. As $M(G,V)=|G|/3$, we have $|C_G(v_1)\cap C_G(w)|=|C_G(v_1+w)|\ge3$. On the other hand, $C_G(v_1)\cap C_G(w)\cap C=C_C(w)=1$, so $C_G(v_1)\cap C_G(w)$ embeds into $C_G(v_1)/C$, which has order $3$ by (a). Hence $C_G(w)C\supseteq C_G(v_1)$. As this holds for every $v_1\in V_1^\#$ and the groups $C_{\bar G}(v_1)=C_G(v_1)/C$ generate $\bar G$ by Lemma~\ref{exc}, we obtain $C_G(w)C=G$. Since $C_G(w)\cap C=1$, it follows that $C_G(w)\cong\bar G$, and so $|C_G(w)'|=|\bar G'|=7=|Q|$. As $C_G(w)'\le G'=Q$, we get $Q\le C_G(w)$. So $Q$ fixes every $w\in W$ with $C_C(w)=1$, and as these vectors span $W$ by Lemma~\ref{abelian}(c), the claim follows.

\smallskip\noindent\emph{(c) $G=C\times D$ with $D=C_G(W)$.} By (b), $G'=Q\le D$, so $G/D$ is abelian. As $G/D$ acts faithfully and completely reducibly on $W$, by Lemma~\ref{abelian}(b) there is $w_0\in W$ with $C_G(w_0)=D$. Let $v_1\in V_1^\#$. Then
\begin{align*}
|(v_1+w_0)^G|&=|G:C_G(w_0)|\cdot|C_G(w_0):C_G(w_0)\cap C_G(v_1)|\\
&=|G:D|\cdot|D:C_D(v_1)|=\frac{|G|}{|C_D(v_1)|},
\end{align*}
and as $M(G,V)=|G|/3$, we get $|C_D(v_1)|\ge3$. Since $D\cap C=C_G(V)=1$, the natural map $D\to\bar G$ is injective and maps $C_D(v_1)$ into $C_{\bar G}(v_1)$, which has order $3$. Hence $DC/C\supseteq C_{\bar G}(v_1)$ for all $v_1\in V_1^\#$, and so $DC/C=\bar G$ by Lemma~\ref{exc}. Therefore $G=DC$, and as $C$ and $D$ are normal subgroups with $C\cap D=1$, we have $G=C\times D$.

\smallskip
Now $D\cong\bar G$ acts faithfully on $V_1$, and $(D,V_1)$ is isomorphic to $(\bar G,V_1)$, hence to some $(D_i,U_i)$. Moreover, $D$ acts trivially on $W$, and the abelian group $C$ acts trivially on $V_1$ and faithfully on $W$. Thus $(G,V)$ is exceptional (with $U=V_1$ and $A=C$), which completes the proof.
\end{proof}

\begin{proof}[Proof of Theorem~\ref{thmA}]
As $p$ is the smallest prime divisor of $|G|$, the group $G$ is $p$-rough. If $G$ is nilpotent, then $\rho(G,V)\ge p$ by Proposition~\ref{nilpotent}, and $(G,V)$ is not exceptional, as $D_1$ and $D_2$ are not nilpotent. If $G$ is not nilpotent, then by Theorem~\ref{cr} either $\rho(G,V)\ge p$, or $(G,V)$ is exceptional, $p=3$ and $\rho(G,V)=7/3$. This proves the first assertion, and the second one follows from Proposition~\ref{Gprime}.
\end{proof}

\begin{proof}[Proof of Corollary~\ref{corB}]
Let $p$ be the smallest prime divisor of $|G|$. If $(G,V)$ is not exceptional, then $|G:G'|\le M/p\le M/3<\frac37M$ by Theorem~\ref{thmA}. If $(G,V)$ is exceptional, then $|G:G'|=\frac37M$ by Theorem~\ref{cr}.
\end{proof}

\section{Modules of odd order}\label{sec:oddV}

In this section we prove Theorem~\ref{thmC}. The following result is our main tool for dealing with imprimitive modules. Here the assumptions that $|G|$ and the characteristic of $V$ are odd are essential.

\begin{proposition}\label{imprim}
Let $G$ be a finite group of odd order and let $V$ be a finite $G$-module over a field of odd characteristic. Suppose that $V=W_1\oplus\cdots\oplus W_n$ for subspaces $W_1,\ldots,W_n$ which are permuted transitively by $G$, and let $K$ be the kernel of the action of $G$ on $\{W_1,\ldots,W_n\}$. Then
\[
M(G,V)\ge|G:K|\cdot M(K,V).
\]
\end{proposition}

\begin{proof}
Let $H=N_G(W_1)$ be the stabilizer of $W_1$ in $G$, and choose $t_1=1,t_2,\ldots,t_n\in G$ with $W_1t_i=W_i$ for all $i$. By Lemma~\ref{minus}, applied to the action of $H$ on $W_1$, there is an $H$-invariant subset $S\subseteq W_1^\#$ with $W_1^\#=S\cup(-S)$ and $S\cap(-S)=\emptyset$. For $1\le i\le n$ put $S_i=St_i\subseteq W_i^\#$. Then $W_i^\#=S_i\cup(-S_i)$ and $S_i\cap(-S_i)=\emptyset$. Moreover, if $g\in G$ and $W_ig=W_j$, then $t_igt_j^{-1}\in H$, and therefore
\begin{equation}\label{transport}
S_ig=S(t_igt_j^{-1})t_j=St_j=S_j.
\end{equation}
Since $G/K$ acts faithfully on the set $\{W_1,\ldots,W_n\}$ as a permutation group of odd order, Theorem~\ref{Gluck} yields a subset $\Delta\subseteq\{1,\ldots,n\}$ such that, for $\mathcal W_\Delta=\{W_i\mid i\in\Delta\}$, every $g\in G$ with $\mathcal W_\Delta g=\mathcal W_\Delta$ lies in $K$. Put $\varepsilon_i=1$ if $i\in\Delta$, and $\varepsilon_i=-1$ otherwise.

Now choose $x\in V$ with $|x^K|=M(K,V)$, and write $x=x_1+\cdots+x_n$ with $x_i\in W_i$. As $K$ leaves every $W_i$ invariant, we have $C_K(x)=\bigcap_iC_K(x_i)$. Thus replacing any $x_i=0$ by an arbitrary element of $W_i^\#$ does not enlarge $C_K(x)$, and so we may assume that $x_i\ne0$ for all $i$. Choose $\delta_i\in\{1,-1\}$ with $\delta_ix_i\in S_i$, and put
\[
v=\sum_{i=1}^n\varepsilon_i\delta_ix_i.
\]
Since $C_K(-x_i)=C_K(x_i)$ for all $i$, we have $C_K(v)=C_K(x)$.

Let $g\in C_G(v)$, and let $\pi$ be the permutation of $\{1,\ldots,n\}$ with $W_ig=W_{\pi(i)}$. Comparing the components in $W_{\pi(i)}$ of $vg=v$, we obtain $\varepsilon_i\delta_ix_ig=\varepsilon_{\pi(i)}\delta_{\pi(i)}x_{\pi(i)}$ for all $i$. By \eqref{transport} we have $\delta_ix_ig\in S_ig=S_{\pi(i)}$, and also $\delta_{\pi(i)}x_{\pi(i)}\in S_{\pi(i)}$. As $S_{\pi(i)}\cap(-S_{\pi(i)})=\emptyset$, it follows that $\varepsilon_i=\varepsilon_{\pi(i)}$ for all $i$. Therefore $\mathcal W_\Delta g=\mathcal W_\Delta$, and hence $g\in K$. This shows that $C_G(v)=C_K(v)=C_K(x)$, and consequently
\[
M(G,V)\ge|v^G|=|G:C_K(x)|=|G:K|\cdot|K:C_K(x)|=|G:K|\cdot M(K,V).
\]
\end{proof}

The heart of the matter is the following result on non-nilpotent groups.

\begin{theorem}\label{nonnil}
Let $p$ be an odd prime, let $G$ be a non-nilpotent $p$-rough group, and let $V$ be a finite faithful completely reducible $G$-module, possibly of mixed characteristic, such that $|V|$ is odd. Then
\[
M(G,V)\ge\frac{13}{3}|G:G'|,
\]
and if $p\ge5$, then $M(G,V)>2p|G:G'|$.
\end{theorem}

\begin{proof}
We argue by induction on $|G|$. We have to show that $\rho(G,V)\ge13/3$, and that $\rho(G,V)>2p$ if $p\ge5$; we refer to this as the assertion for $(G,V)$. We proceed in several steps.

\smallskip\noindent\emph{Step 1. If $G$ has a regular orbit on $V$, then the assertion holds.} Indeed, then $M(G,V)=|G|=|G'|\cdot|G:G'|$, so that $\rho(G,V)=|G'|\ge2p+1$ by Lemma~\ref{comm}(b), and $2p+1\ge7>13/3$.

\smallskip\noindent\emph{Step 2. A reduction to the case that $V$ is irreducible.} Write $V=V_1\oplus\cdots\oplus V_t$ with irreducible $G$-modules $V_i$. As $G$ is isomorphic to a subgroup of the direct product of the groups $G/C_G(V_i)$ and $G$ is not nilpotent, some $G/C_G(V_i)$ is not nilpotent, and we may assume that this holds for $i=1$. Put $C=C_G(V_1)$, $\bar G=G/C$ and $W=V_2\oplus\cdots\oplus V_t$. Then $C$ acts faithfully on $W$, and $W$ is a completely reducible $C$-module by Clifford's theorem, so that $M(C,W)\ge|C:C'|$ by Theorem~\ref{KY}. Suppose that $C\ne1$. Then $|\bar G|<|G|$, and $\bar G$ is a non-nilpotent $p$-rough group, so that the assertion holds for $(\bar G,V_1)$ by induction. As $M(G,V_1)=M(\bar G,V_1)$, Lemma~\ref{product}, Theorem~\ref{KY},  and Lemma~\ref{quot} yield
\[
\rho(G,V)\ge\rho(\bar G,V_1)\cdot\frac{M(C,W)}{|C:C'|}\cdot\frac{|\bar G:\bar G'|\cdot|C:C'|}{|G:G'|}\ge\rho(\bar G,V_1),
\]
and the assertion holds for $(G,V)$. If $C=1$, then $G$ acts faithfully on $V_1$, and $\rho(G,V)\ge\rho(G,V_1)$. Hence it suffices to prove the theorem in the case that $V$ is irreducible, which we assume from now on.

\smallskip\noindent\emph{Step 3. The case that $V$ is quasiprimitive.} By Lemma~\ref{Espuelas}, either $G$ has a regular orbit on $V$, and we are done by Step 1, or $F(G)$ is cyclic. In the latter case, $V_{F(G)}$ is homogeneous, as $V$ is quasiprimitive, and Proposition~\ref{semi}(a) yields $v\in V$ with $|C_G(v)|\le\frac3{13}|G'|$, and with $|C_G(v)|<\frac1{2p}|G'|$ if $p\ge5$. As $\rho(G,V)\ge|G:C_G(v)|/|G:G'|=|G'|/|C_G(v)|$, the assertion follows.

\smallskip\noindent\emph{Step 4. The case that $V$ is not quasiprimitive.} Then there exists $N\nsg G$ such that $V_N$ is not homogeneous. By Clifford's theorem, $V_N=W_1\oplus\cdots\oplus W_n$ with $n\ge2$, where $W_1,\ldots,W_n$ are the homogeneous components of $V_N$, and $G$ permutes $W_1,\ldots,W_n$ transitively. Let $K$ be the kernel of this permutation action. Then $K\nsg G$, and $K<G$ as $n\ge2$. By Proposition~\ref{imprim} we have
\[
M(G,V)\ge|G:K|\cdot M(K,V).
\]
Moreover, $K$ is $p$-rough and acts faithfully on $V$, and $V$ is a completely reducible $K$-module by Clifford's theorem. If $K$ is nilpotent, then $M(K,V)=|K|$ by Lemma~\ref{nilreg}, so that $M(G,V)\ge|G|$. Thus $G$ has a regular orbit on $V$, and we are done by Step 1. If $K$ is not nilpotent, then the assertion holds for $(K,V)$ by induction, and Lemma~\ref{quot} yields
\begin{align*}
\rho(G,V)&\ge\rho(K,V)\cdot\frac{|G:K|\cdot|K:K'|}{|G:G'|}\\
&\ge\rho(K,V)\cdot\frac{|G/K:(G/K)'|\cdot|K:K'|}{|G:G'|}\ge\rho(K,V).
\end{align*}
So the assertion holds for $(G,V)$, and the proof is complete.
\end{proof}

\begin{remark}\label{pairs}
Let $G$ and $V$ be as in Theorem~\ref{thmC}. By Lemma~\ref{minus}, for $v\in V^\#$ the orbits $v^G$ and $(-v)^G$ are distinct and have the same size. In particular, the orbits of $G$ of maximal size on $V^\#$ come in pairs, and if $V\ne0$, then in the situation of Theorem~\ref{thmC}(a) the group $G$ has at least two regular orbits on $V$. This is in accordance with \cite[Theorem 2.1(c)]{CK24b}, by which a nonabelian group $G$ with $|G:G'|=M$ has at least two orbits of size $M$.
\end{remark}

We are now ready to prove Theorem~\ref{thmC} and Corollary~\ref{corD}.

\begin{proof}[Proof of Theorem~\ref{thmC}]
(a) By Lemma~\ref{nilreg}, $G$ has a regular orbit on $V$. So $M=|G|$, and $|G:G'|=|G|/|G'|=M/|G'|$.

(b) By Theorem~\ref{nonnil} we have $M\ge\frac{13}{3}|G:G'|$, and $M>2p|G:G'|$ if $p\ge5$. Choose $v\in V$ with $|G:C_G(v)|=M$. Then $|C_G(v)|=|G|/M\le\frac3{13}|G'|$, and $|C_G(v)|<\frac1{2p}|G'|$ if $p\ge5$. The remaining assertions are clear.

(c) If $G$ is nilpotent, then $|G:G'|=M/|G'|$ by (a). As $G$ is nonabelian, $G'\ne1$, so $|G'|\ge p$, and hence $|G:G'|\le M/p$, with equality if and only if $|G'|=p$. If $G$ is not nilpotent, then $|G'|\ne p$ by Lemma~\ref{comm}(a), and by (b) we have $|G:G'|\le\frac3{13}M<\frac M3=\frac Mp$ if $p=3$, and $|G:G'|<\frac M{2p}<\frac Mp$ if $p\ge5$.
\end{proof}

\begin{proof}[Proof of Corollary~\ref{corD}]
Suppose that $|G'|=p$. By Lemma~\ref{comm}(a), $G'\le Z(G)$ and $G$ is nilpotent, so $G=P\times A$ with $P\in\Syl_p(G)$ and $A$ the Hall $p'$-subgroup of $G$. As $G'=P'\times A'$, we get $A'=1$ and $|P'|=p$. Conversely, if $G=P\times A$ with $P$ and $A$ as in the statement, then $G'=P'$ has order $p$.

Now let $|G'|=p$. Then $P'=G'\le Z(G)$, so $P$ has nilpotency class $2$, and $[x^p,y]=[x,y]^p=1$ for all $x,y\in P$. Hence $P/Z(P)$ is elementary abelian. As $P'\le Z(P)$, the commutator map induces an alternating bilinear form on the $\GF(p)$-vector space $P/Z(P)$ with values in $P'\cong\GF(p)$, and this form is nondegenerate by the definition of $Z(P)$. Therefore $|P:Z(P)|=p^{2m}$ for some $m$, and $m\ge1$ as $P$ is nonabelian. By Proposition~\ref{Gprime}, $G$ has a regular orbit on every finite faithful completely reducible $G$-module $V$, and $|G:G'|=M/p$.

Finally, let $V$ be a faithful irreducible $G$-module over a finite field $F$. By Schur's lemma, $D=\End_{FG}(V)$ is a finite division ring and hence a field, and $Z(G)$ acts on $V$ by elements of $D^\times$. As $V$ is faithful, $Z(G)$ is isomorphic to a subgroup of $D^\times$ and therefore is cyclic. Since $A\le Z(G)$, also $A$ is cyclic.
\end{proof}

\section{Examples}\label{sec:examples}

In this section we present the examples mentioned in the introduction.

\begin{example}\label{ex:gamma}
Let $p$ be an odd prime, let $E=\GF(2^p)$, and let $G=\Gamma(E)$ act on $V=E\cong V(p,2)$. If $q$ is a prime divisor of $2^p-1$, then the order of $2$ modulo $q$ is $p$, so that $q\equiv1\pmod p$; hence $p$ is the smallest prime divisor of $|G|=p(2^p-1)$. The automorphism $x\mapsto x^2$ has fixed field $\GF(2)$ and therefore acts fixed point freely on $\Gamma_0(E)\cong C_{2^p-1}$. As in step (2) of the proof of Proposition~\ref{semi}, it follows that $G'=\Gamma_0(E)$ and $|G:G'|=p$. As $\Gamma_0(E)$ acts transitively on $V^\#$, we have $M=2^p-1$, and therefore
\[
|G:G'|=\frac{p}{2^p-1}\,M.
\]
For $p=3$ we obtain $(G,V)=(D_1,U_1)$ and $|G:G'|=\frac37M>\frac M3$. For $p\ge5$ we have $2^p-1>p^2$, and so $|G:G'|<M/p$, in accordance with Theorem~\ref{thmA}. So the reason for the existence of exceptional configurations for $p=3$ is the inequality $2^3-1<3^2$.
\end{example}

\begin{example}\label{ex:13}
Let $E=\GF(27)$, let $A\le E^\times$ be the subgroup of order $13$, which consists of the nonzero squares of $E$, and let $\sigma\in\Gal(E)$ be the Frobenius automorphism $x\mapsto x^3$. Let $G=A\langle\sigma\rangle\le\Gamma(E)$, where $a\in A$ acts as $x\mapsto ax$. Then $\sigma$ acts on $A$ by $a\mapsto a^3$, and as $3$ has order $3$ modulo $13$, this action is fixed point free. So $G\cong C_{13}\rtimes C_3$ is a Frobenius group of order $39$ with $G'=A$, and $|G:G'|=3$. Moreover, $A$ acts irreducibly on $V=E\cong V(3,3)$, again because $3$ has order $3$ modulo $13$, and so $V$ is a faithful irreducible quasiprimitive $G$-module. The group $A$ has exactly two orbits on $V^\#$, namely $A$ and $-A$ (note that $-1$ is not a square in $E$), and both are $\sigma$-invariant, as $x^3$ is a square if and only if $x$ is. Hence $G$ has exactly two orbits on $V^\#$, each of size $13$. Thus $M=13$ and
\[
|G:G'|=3=\frac{3}{13}M,
\]
so the bound in Theorem~\ref{thmC}(b) is best possible. Of course, $G$ has no regular orbit on $V$, as $|G|=39>|V^\#|=26$.

By Lemma~\ref{Espuelas} and Proposition~\ref{semi}(a), if equality holds in Theorem~\ref{thmC}(b) and $V$ is quasiprimitive, then $|V|=27$ and $|G|=39$. Further examples of equality are obtained as follows: if $B$ is an abelian group of odd order and $W$ is a finite faithful completely reducible $B$-module with $|W|$ odd, then $G\times B$ acts faithfully and completely reducibly on $V\oplus W$, with $G$ acting trivially on $W$ and $B$ acting trivially on $V$. Here $M(G\times B,V\oplus W)=13|B|$, as $B$ has a regular orbit on $W$ by Lemma~\ref{abelian}(b), and so again $|G\times B:(G\times B)'|=3|B|=\frac{3}{13}M(G\times B,V\oplus W)$.
\end{example}

\begin{example}\label{ex:SG}
Let $p\ge5$ be a prime such that $q=2p+1$ is a prime; for instance, $p\in\{5,11,23,29\}$. Then $p\not\equiv1\pmod3$, since otherwise $3$ would divide $q$, and so $p\equiv5\pmod 6$ and $q\equiv11\pmod{12}$. By quadratic reciprocity, $3$ is a square modulo $q$, so the order of $3$ modulo $q$ divides $(q-1)/2=p$, and as $3\not\equiv1\pmod q$, this order equals $p$. Let $E=\GF(3^p)$, let $A\le E^\times$ be the subgroup of order $q$, let $\sigma\in\Gal(E)$ be the automorphism $x\mapsto x^3$, and let $G=A\langle\sigma\rangle\le\Gamma(E)$. Then $\sigma$ acts on $A$ by $a\mapsto a^3$, i.e., fixed point freely, so that $G\cong C_q\rtimes C_p$ is a Frobenius group with $G'=A$ and $|G:G'|=p$, and $p$ is the smallest prime divisor of $|G|$. As $3$ has order $p$ modulo $q$, the group $A$ acts irreducibly on $V=E\cong V(p,3)$.

The nonidentity elements of $A$ have no nonzero fixed points on $V$. Every element of $G\setminus A$ acts as $x\mapsto bx^{3^j}$ for some $b\in E^\times$ and some $j\not\equiv0\pmod p$, and as the fixed field of $x\mapsto x^{3^j}$ is $\GF(3)$, the argument in step (5) of the proof of Proposition~\ref{semi} shows that such an element has at most two nonzero fixed points on $V$. As $G$ has exactly $q$ subgroups of order $p$, at most $2q$ vectors of $V^\#$ have a nontrivial stabilizer in $G$. Since $2q<3^p-1$, the group $G$ has a regular orbit on $V$, whence $M=|G|=pq$ and
\[
|G:G'|=\frac{M}{2p+1}.
\]
So the bound $|G:G'|<M/(2p)$ in Theorem~\ref{thmC}(b) cannot be improved to $|G:G'|<M/(2p+1)$ for such $p$. For $p=5$ we obtain the group $C_{11}\rtimes C_5$ acting on $V(5,3)$; a direct computation shows that here exactly $22$ vectors of $V^\#$ have a nontrivial stabilizer, so that $G$ has exactly four regular orbits on $V$.
\end{example}

\begin{example}\label{ex:heis}
Let $p$ be an odd prime, let $s$ be a prime power with $s\equiv1\pmod p$, and let $\omega\in\GF(s)$ be a primitive $p$-th root of unity. Let $e_0,\ldots,e_{p-1}$ be the standard basis of $V=V(p,s)$, with indices taken modulo $p$, and let $G=\langle x,y\rangle\le\GL(V)$, where $e_ix=\omega^ie_i$ and $e_iy=e_{i+1}$ for all $i$. Then $x^p=y^p=1$ and $y^{-1}xy=\omega^{-1}x$, so that $[x,y]=\omega^{-1}\cdot1_V$, and $G=(\langle x\rangle\times\langle\omega\cdot1_V\rangle)\rtimes\langle y\rangle$ is the Heisenberg group of order $p^3$ and exponent $p$, with $G'=Z(G)=\langle\omega\cdot1_V\rangle$ of order $p$. The eigenspaces of $x$ on $V$ are the coordinate axes, which afford pairwise distinct eigenvalues and are permuted transitively by $y$; hence $V$ is a faithful irreducible $G$-module. By Corollary~\ref{corD}, $G$ has a regular orbit on $V$, so that $M=p^3$ and
\[
|G:G'|=p^2=\frac Mp,
\]
in accordance with Theorems~\ref{thmA} and~\ref{thmC}. Note that $s$ may be odd (for instance, any prime $s\equiv1\pmod p$) or even (for instance, $s=2^f$, where $f$ is the order of $2$ modulo $p$).
\end{example}

\begin{example}\label{ex:wr}
Let $V=V(3,4)$, let $B\le\GL(3,4)$ be the group of all diagonal matrices, which is elementary abelian of order $27$, and let $u\in\GL(3,4)$ be the permutation matrix of a $3$-cycle. Then $G=B\langle u\rangle\cong C_3\wr C_3$ is a $3$-group of order $81$ acting faithfully on $V$, and $V$ is irreducible, since the coordinate axes afford pairwise distinct linear characters of $B$ and are permuted transitively by $u$. The map $b\mapsto b^{-1}b^u$ is a homomorphism from $B$ onto $[B,u]$ whose kernel $C_B(u)$ consists of the scalar matrices in $B$, so that $|[B,u]|=9$. As in Step (2) of the proof of Proposition~\ref{semi}, $[B,u]\nsg G$ and $G/[B,u]$ is abelian, and so $G'=[B,u]$ has order $9$ and $|G:G'|=9$. If $v\in V$ has three nonzero coordinates, then $v^B$ consists of all $27$ vectors with three nonzero coordinates, as $\GF(4)^\times$ has order $3$. If $v$ has exactly two nonzero coordinates, then $|v^B|=9$ and $|v^G|=27$, and if $v$ has exactly one nonzero coordinate, then $|v^G|=9$. Hence
\[
M=27=3|G:G'|,
\]
although $|G'|=9$, and $G$ has no regular orbit on $V$ (as also $|G|=81>|V|=64$). So without the hypothesis that $|V|$ is odd, the characterization of equality in Theorem~\ref{thmC}(c) fails for $p=3$, even for nilpotent groups. Note that $3=2^2-1$ is a Mersenne prime, so that $G$ is one of the exceptions in Theorem~\ref{Yang}.
\end{example}

\section{Outlook}\label{sec:outlook}

We conclude with some remarks on possible future research related to our results.

First, Theorems~\ref{thmA} and~\ref{thmC} leave open the characterization of the case of equality $|G:G'|=M/p$ if $|V|$ is even. For $p=3$, equality occurs for the nilpotent group $C_3\wr C_3$ with $|G'|=9$ (Example~\ref{ex:wr}). In general, a nilpotent group $G$ with $|G:G'|=M/p$ and $|G'|>p$ cannot have a regular orbit on $V$, since otherwise $M=|G|$ and $M/|G:G'|=|G'|$. Hence, by Lemma~\ref{nilreg}, such a group involves $C_r\wr C_r$ for some Mersenne prime $r\ge p$. For instance, if $r\ge7$ is a Mersenne prime, then $C_r\wr C_r$ acting monomially on $V(r,r+1)$ satisfies $M=r^r$ and $|G:G'|=r^2$, so that $M/|G:G'|=r^{r-2}>r$; this suggests that for $p\ge5$ there might be no such examples.

\begin{question}\label{q1}
Let $G$ be a finite nonabelian group of odd order, let $p$ be the smallest prime divisor of $|G|$, and let $V$ be a finite faithful completely reducible $G$-module of even order such that $|G:G'|=M/p$. Is $G$ nilpotent? If $p\ge5$, is it true that $|G'|=p$?
\end{question}

Second, in the situation of Theorem~\ref{thmC}(b), the quotient $M/|G:G'|$ is at least $13/3$ if $p=3$, where this bound is attained (Example~\ref{ex:13}), and it is larger than $2p$ if $p\ge5$. By Example~\ref{ex:SG}, the latter bound is close to best possible if $2p+1$ is a prime.

\begin{question}\label{q2}
Let $p\ge5$ be a prime. What is the infimum $c_p$ of the numbers $M/|G:G'|$, where $G$ runs through the non-nilpotent groups of odd order with smallest prime divisor $p$ and $V$ through the finite faithful completely reducible $G$-modules of odd order? By Theorem~\ref{thmC} and Example~\ref{ex:SG}, we have $c_p\ge2p$, and $c_p\le2p+1$ if $2p+1$ is a prime. Moreover, which pairs $(G,V)$ with $p=3$ satisfy $|G:G'|=\frac3{13}M$? Are they all of the form described in Example~\ref{ex:13}?
\end{question}

Third, in \cite[Theorem C]{CK24b} faithful $G$-modules $V$ of characteristic $s$ with $O_s(G)=1$ which are not necessarily completely reducible were also considered. If $|G|$ is odd and $s=2$, then the action of $G$ on $V$ is coprime, so that $V$ is completely reducible by Maschke's theorem and our results apply. For odd $s$, however, complete reducibility is used in an essential way in our arguments.

\begin{question}\label{q3}
Let $G$ be a finite nonabelian group of odd order, let $p$ be the smallest prime divisor of $|G|$, and let $V$ be a finite faithful $G$-module of odd characteristic $s$ with $O_s(G)=1$. Is it true that $|G:G'|\le M(G,V)/p$?
\end{question}

Finally, it would be interesting to know to what extent the results of this paper extend to groups of even order. Here the classification in \cite{CK24b} shows that $|G:G'|=M$ can occur for nonabelian groups, but one may ask for the best possible bound for nonabelian groups $G$ which do not have a direct factor isomorphic to one of the groups occurring in \cite[Theorems A and B]{CK24b}.
\section*{}
\begin{center}
    \textbf{Acknowledgement}
\end{center}
 Thie author was supported by a grant from the Simons Foundation (SFI-MPS-TSM-00014153 to Thomas M. Keller).

\begin{center}
    \textbf{AI statement}
\end{center}

This project was started a few years ago, and the author knew most of the results of this paper before any AI use. With the availability of AI, the author finished the paper using Claude by Anthropic. After using this tool, the author reviewed and edited the content as needed and takes full responsibility for the final mathematical and narrative integrity of the publication.


\end{document}